\documentclass{amsart}

\usepackage{amsfonts}
\usepackage{amssymb}
\usepackage{latexsym}

\newtheorem{theorem}{Theorem}[section]
\newtheorem{lemma}[theorem]{Lemma}
\newtheorem{proposition}[theorem]{Proposition}

\theoremstyle{remark}
\newtheorem{definition}[theorem]{Definition}
\newtheorem{example}[theorem]{Example}
\newtheorem{remark}[theorem]{Remark}

\newcommand{\C}{\mathbb{C}}
\newcommand{\Q}{\mathbb{Q}}
\newcommand{\R}{\mathbb{R}}
\newcommand{\Z}{\mathbb{Z}}

\newcommand{\Sp}{\mathbb{S}}
\newcommand{\Imt}{\operatorname{Im}}
\newcommand{\co}{\colon\thinspace}
\newcommand{\T}{\mathbb{T}}
\newcommand{\Pb}{\mathbb{P}}

\begin{document}

\title{Functional Equations for Orbifold Wreath Products}

\author{Carla Farsi}
\address{Department of Mathematics,
University of Colorado at Boulder,
Campus Box 395,
Boulder, CO 80309-0395}
\email{farsi@euclid.colorado.edu}

\author{Christopher Seaton}
\address{Department of Mathematics and Computer Science,
Rhodes College,
2000 N. Parkway,
Memphis, TN 38112}
\email{seatonc@rhodes.edu}



\begin{abstract}
We present generating functions for extensions of multiplicative invariants of
wreath symmetric products of orbifolds presented as the quotient by the locally
free action of a compact, connected Lie group in terms of orbifold sector decompositions.
Particularly interesting instances of these product formulas occur for the Euler and
Euler--Satake characteristics, which we compute for a class of weighted projective spaces.
This generalizes results known for global quotients by
finite groups to all closed, effective orbifolds.  We also describe a combinatorial
approach to extensions of multiplicative invariants using decomposable functors that recovers
the formula for the Euler--Satake characteristic of a wreath product of a global quotient orbifold.
\end{abstract}

\maketitle


\section{Introduction}
\label{sec-intro}

If $M$ is a manifold, the $n$th \emph{symmetric product} of $M$ is the quotient of
the Cartesian product $M^n$ by the action of the symmetric group by permuting factors.
Characteristic numbers of symmetric products of manifolds have been widely studied,
and their structure naturally leads to generating functions of the
infinite product and exponential type, see e.g. \cite{hh, oh, dmvv}.
If $M$ is in addition equipped with the action of a finite group
so that $M/G$ is a global quotient orbifold, the \emph{wreath symmetric product}
is the natural generalization of the symmetric product.  In the literature one can find several approaches
to characteristic numbers of wreath symmetric products, see e.g. \cite{tamanoi1, tamanoi2, tamanoi3, wang, wangZhou}.
Many examples of orbifolds, however, are not global quotients of a manifold by a finite group.
The most well-known examples are the weighted complex projective spaces, see e.g.
\cite{kawasaki, planes, bfnr, holm, jiang}.  Characteristic numbers of wreath symmetric products for
non-global quotient orbifolds have been studied e.g. in
\cite{boli, hama, hao, farsiseaton2, farsiseaton3}.

In \cite{tamanoi1} and \cite{tamanoi2}, Tamanoi introduced a number
of orbifold invariants for global quotient orbifolds, i.e. orbifolds given by
the quotient of a manifold by a finite group, generalizing the
orbifold Euler characteristics of \cite{bryanfulman} and \cite{dixon}.
The basic idea behind these invariants is to
apply a multiplicative orbifold invariant $\varphi$, e.g. the
Euler--Satake characteristic (see \cite{farsiseaton3}), to a $\Gamma$-sector decomposition of the
orbifold, yielding an extension $\varphi_{\Gamma}$ of this invariant. Tamanoi
introduced sector decompositions of global quotients associated to
an arbitrary group $\Gamma$, a $\Gamma$-set $X$, and a finite
covering space $\Sigma^\prime \to \Sigma$ of a connected manifold
$\Sigma$ with fundamental group $\Gamma$. See also \cite{tamanoi3} for connections between these
extensions and the orbifold elliptic genus, which for non global quotient orbifolds was
introduced in \cite{dlm}.

In \cite{farsiseaton1}, for a finitely generated discrete group
$\Gamma$, the authors introduced  the $\Gamma$-sectors associated
to a Lie groupoid $\mathcal{G}$, which generalized Tamanoi's
$\Gamma$-sector decomposition to the case of an arbitrary orbifold.
The relationship between this construction, quotient presentations
of orbifolds, and generalized loop spaces for orbifolds was studied
in \cite{farsiseaton2}.  In \cite{farsiseaton3}, this relationship
was used to extend Tamanoi's generating functions for the extension
by free abelian groups $\Gamma = \Z^{\ell}$ of the Euler and Euler--Satake
characteristics of wreath symmetric product orbifolds to the case of
closed effective orbifolds by using their presentation
as the quotient of a closed manifold by the locally free action
of a compact, connected Lie group \cite[Theorem 2.19]{moerdijkmrcun}.
See also \cite{wang} for the generating function for the stringy orbifold
Euler characteristic, corresponding to $\Gamma = \Z^2$.

In this paper, we  generalize Tamanoi's generating function
for the $\Gamma$-extension $\varphi_{\Gamma}$ of a
wreath product of a multiplicative invariant $\varphi$
to the case of quotients by compact, connected Lie groups acting locally
freely; see Theorem \ref{th-GeneratingFunctionGammaExtensionsFuncEqnNewTheo}.
This generalizes the generating functions given in \cite{farsiseaton3}
from the case where $\varphi$ is the Euler or Euler--Satake
characteristic and $\Gamma = \Z^\ell$ to general $\varphi$
and an arbitrary finitely generated discrete group $\Gamma$.
For specific multiplicative invariants $\varphi$, these formulas relate the values of an
extension $\varphi_{\Gamma}$ of $\varphi$ on the wreath symmetric products
$ M^n \rtimes G(\mathcal{S}_n)$ to  values of extensions $\varphi_{H}$  of $\varphi$ on $M\rtimes G$; see Theorems
\ref{th-GenerFunctGammaExtensStandardEulerCharacteristic} and
\ref{th-GenerFunctGammaExtensEulerSatakeCharacteristic}.
These expansion formulas admit both an infinite product and an exponential form.
For a geometric interpretation, note that if $\Sigma$ is a manifold with fundamental groups $\Gamma$,
then each factor in the  infinite product corresponds to a connected covering space
associated to a finite index subgroup $H$ of $\Gamma$.
This generalization requires defining sector decompositions
associated to $\Gamma$-sets in this generality.
We illustrate these results by calculating the $\Gamma$-Euler and $\Gamma$-Euler--Satake
characteristics, where $\Gamma$ is the fundamental group a closed,
orientable surface of positive genus,
of wreath symmetric products of an interesting class
of orbifold examples that are not global quotients: weighted projective spaces with weights $(m, mn, n)$
where $m$ and $n$ are relatively prime.

In the last section, emphasizing the exponential form of the wreath product expansions,
we introduce for global quotient orbifolds a modification of methods of Dress and M\"uller
\cite{dressmuller} for decomposable functors  to
relate the $\Gamma$-extension of the Euler--Satake characteristic
and the $(\Gamma/H)$-extensions, see Theorem \ref{th-DressMuellerGen}.
Indeed their methods  provide a counting algorithm for invariants of exponential type,
and using this method, we recover Theorem \ref{th-GenerFunctGammaExtensEulerSatakeCharacteristic}
for global quotients.  Late in the preparation of this paper, the authors became aware of
\cite{dressmullerII}, from which the results of this section also follow by choosing the
Euler--Satake characteristic as a weight.  The method of proof is similar to those in
\cite{dressmullerII}.

The outline of this paper is as follows.  In Section \ref{sec-GammaGBundlTheClassifCrosWrProd},
we collect background material on $K$-$G$-bundles for groups $K$ and $G$
and review the classifications given in \cite{tamanoi2}.
Wreath products appear naturally in the geometric context of $K$-$G$-bundles, which are hence
a standard tool to study homomorphisms into wreath products.
In Section \ref{sec-FunctionalEquationsforQuotOrbifoldWreathProducts},
we generalize Tamanoi's generating functions for $\Gamma$-extensions of
multiplicative invariants of wreath symmetric products to orbifolds
presented as quotients by compact, connected Lie groups acting
locally freely; see Theorem
\ref{th-GeneratingFunctionGammaExtensionsFuncEqnNewTheo}.
This provides a geometric context for the previous section.
We apply these results to the Euler and Euler--Satake characteristics in
Subsection \ref{subsec-ExamplesTamanoiFunctEqautionWreathProducts},
resulting in generating functions for these invariants given in
Theorems \ref{th-GenerFunctGammaExtensStandardEulerCharacteristic}
and \ref{th-GenerFunctGammaExtensEulerSatakeCharacteristic}.
We illustrate Theorems \ref{th-GenerFunctGammaExtensStandardEulerCharacteristic}
and \ref{th-GenerFunctGammaExtensEulerSatakeCharacteristic} in
Section \ref{subsec-wrExamplesExpansionExamplesIIItheHinertia}
for the class of non-global quotient orbifolds given by weighted projective spaces with weights $(m, mn, n)$
with $n>1$ and $m>1$ relatively prime and $\Gamma$
the fundamental group of a closed orientable two surface of positive genus;
the standard non-global quotient
$2$-dimensional teardrops $\Pb(1, n)$ and  $\Pb(1, m)$ appear as sectors of these spaces.
We then study extensions of invariants associated to
arbitrary finite $\Gamma$-sets.  In Section
\ref{sec-DecomposableFunctorsAndWreathProducts},
we recover  Theorem \ref{th-GenerFunctGammaExtensEulerSatakeCharacteristic}
using a formal functional equation of Dress and M\"uller
\cite{dressmuller} for decomposable functors.

By a \emph{quotient orbifold}, we mean an orbifold that admits a presentation
as a translation groupoid $M \rtimes G$ where $M$ is a smooth manifold and $G$ is a Lie group acting
locally freely in such a way that $M \rtimes G$ is Morita equivalent to an
orbifold groupoid, see \cite{ademleidaruan}.  For brevity, we refer to $M \rtimes G$ as a
\emph{cc-presentation} when in addition $G$ is compact and connected, $M$
is closed, and the action of $G$ on $M$ is effective.  In particular, note that an orbifold
that admits a cc-presentation is compact and does not have boundary in the orbifold sense.
All manifolds, orbifolds, and group actions are assumed smooth.
We use $\chi$ to denote the (usual) Euler characteristic of the orbit space
and $\chi_{ES}$ to denote the Euler--Satake characteristic, see \cite{farsiseaton3}.
Unless stated otherwise, we will always use $M$ to denote a smooth, closed manifold,
$G$ to denote a compact Lie group, and $\Gamma$ to denote a finitely
generated discrete group.


\section{Classifications of $K$-$G$-Bundles and Conjugacy Classes of Homomorphisms}
\label{sec-GammaGBundlTheClassifCrosWrProd}

In this section, we review results on the classifications of $K$-$G$-bundles
parametrized by conjugacy classes of homomorphisms.  We assume
that $G$ is a compact Lie group and $K$ is a topological group.

\begin{definition}[\cite{hambletonhausmann, lashof1, tamanoi2}]
\label{def-GammaGBundleGgroupCase}
Let $X$ be a topological space.
\begin{itemize}
\item[({\it i})]
A $K$-$G$-bundle over $X$ is a locally trivial
$G$-bundle $p\co P \to X$ with left $K$-actions on $P$ and
$X$ such that the projection map $p$ is $K$-equivariant.
\item[({\it ii})]
A $K$-$G$-principal bundle over $X$ is a locally trivial
principal $G$-bundle $p \co P \to X$ that is also a $K$-$G$-bundle
such that
\[
    \gamma (eg) = (\gamma e) g
    \;\;\;
    \forall \gamma \in K, e\in P, g \in G.
\]
In particular, as $P$ is a principal $G$-bundle, $G$ acts on $P$ on the right.
\end{itemize}
Morphisms of $K$-$G$-bundles and $K$-$G$-principal bundles
are bundle morphisms, respectively principal bundle morphisms,
that are $K$-equivariant.  For a given $K$-$G$-bundle or
$K$-$G$-principal bundle $P$,
we let $Aut_{K,\,G}^P$ denote its automorphism group.

A $K$-$G$-bundle or $K$-$G$-principal bundle is \emph{trivial}
when it is a product.  We refer to a $K$-$G$-bundle $P$ as a
\emph{$K$-irreducible $G$-bundle} when
the $K$-action on $X$ is transitive; similarly,
a \emph{$K$-irreducible $G$-principal bundle} is a
$K$-$G$-principal bundle where $K$ acts transitively on $X$.
\end{definition}

By the associated principal bundle construction, every $K$-$G$-bundle
over $X$ induces a $K$-$G$-principal bundle over $X$.
When $K$ and $G$ are compact Lie
and $X$ is completely regular, every
$K$-$G$-bundle over $X$ is locally
trivial by \cite[Corollary 1.5]{lashof1}.  The same holds true if
the bundle is smooth  \cite[Corollary 1.6]{lashof1}.

By \cite[Lemma 3-3 and Lemma 4-1]{tamanoi2},
wreath products occur as the group of $G$-bundle automorphisms of trivial
$G$-bundles over finite sets, and centralizers of homomorphisms into wreath products
occur as the group of $K$-$G$-automorphisms of
$K$-$G$-bundles over finite sets.  In particular, we have the following
generalization of \cite[Lemma 3-3]{tamanoi2}.

\begin{proposition}
\label{prop-WreathPdctIsAutomGp}
The automorphism group of the trivial $G$-bundle
$X \times G \to X$ over a discrete space $X$ is equal to $Map(X, G)\rtimes K$,
where $K$ is the permutation group of $X$ and the $K$-action is given by
\[
    k f(x) = f(k^{-1} x).
\]
\end{proposition}

\begin{proof}
By \cite[Chapter 5, Theorem 1.1]{husemoller}, the group of automorphisms of $X \times G$
as a principal $G$-bundle that restrict to the identity on $X$ is given by
$Map(X, G)$.  Then it is straightforward to check that
every general automorphism is determined by an element of $Map(X, G)$
and a homeomorphism of $X$.
\end{proof}

\begin{remark}
\label{rem-SpecialCaseXisN}
In the case $X = {\mathbf n} = \{ 1, 2, \ldots, n \}$, $K$
is the symmetric group $\mathcal{S}_n$ and
we obtain the standard wreath product
$G(\mathcal{S}_n)$; see \cite[Lemma 3-3]{tamanoi2}.
That is, $G(\mathcal{S}_n)$ is the semidirect product of $G^n$ by the
action of $\mathcal{S}_n$ by permuting factors, so that
the operation is given by
\[
    ((g_1, \ldots, g_n),s)((h_1, \ldots, h_n), t)
    =
    ((g_1h_{s^{-1}(1)}, \ldots, g_n h_{s^{-1}(n)}), st)
\]
for $(g_1, \ldots, g_n), (h_1, \ldots, h_n) \in G^n$ and $s, t \in \mathcal{S}_n$.

\end{remark}

The $K$-$G$-principal bundles over a finite set $X$ of order $n$ are necessarily trivial,
and are classified by conjugacy classes of homomorphisms $\theta\co K \to G(\mathcal{S}_n)$.
Similarly, the $K$-irreducible $G$-principal bundles over $X$ are classified by conjugacy
classes of homomorphisms, as explained by the following.

\begin{theorem}[\cite{tamanoi2}]
\label{th-CrossedPdctBundleChr}
Let $K$ and $G$ be any groups, and let $X$ be a finite set of order $n$.
\begin{itemize}
\item[({\it i})]
There is a bijective correspondence between the sets
\[
    \left\{\hbox{isomorphism classes of
    $K$-$G$-principal bundles over $X$ } \right\}
\]
and
\[
    \hbox{HOM}\left(K, G(\mathcal{S}_n)\right)
    /G(\mathcal{S}_n).
\]
\item[({\it ii})]
There is a bijective correspondence between the sets
\[
    \left\{\hbox{isomorphism classes of
    $K$-irreducible $G$-principal bundles over $X$ } \right\}
\]
and
\[
    \bigsqcup\limits_{(H_n)}
    \hbox{HOM}(H_n, G)/(N_K(H_n)\times G)
\]
where the union is over $K$-conjugacy classes of subgroups $H_n \leq K$ of index $n$,
and the $N_K(H_n)\times G$-action is given by
\begin{equation}
\label{eq-NormTimesGAction}
    (\rho(u,g))(h) \co = g^{-1} \rho(u\, h \, u^{-1})g
\end{equation}
for $(u,g) \in N_K(H_n) \times G$.
\end{itemize}
\end{theorem}

For details on the correspondence described in ({\it i}) see  \cite[pp. 812--814]{tamanoi2}, and for details on
the correspondence in ({\it ii}), see \cite[pp. 815--816]{tamanoi2}.

By \cite[Lemma 4-1]{tamanoi2}, if $\theta\co K \to G(\mathcal{S}_n)$ is a homomorphism
whose $G(\mathcal{S}_n)$-conjugacy class $(\theta)$ corresponds to the isomorphism class of the
$K$-$G$-bundle $P_\theta$, then
\begin{equation}
\label{eq-CG=Aut}
    C_{G({\mathcal S}_n)}(\theta)\cong Aut_{K,\, G}^{P_{\theta}}.
\end{equation}
Note that if $\rho\co H \to G$ is a homomorphism, we use $(\rho)$ to
denote the $G$-conjugacy class of $\rho$ and $[\rho]$ to denote the
$N_K(H)\times G$-conjugacy class of $\rho$.


\section{Functional Equations for Quotient Orbifold Wreath Symmetric Products}
\label{sec-FunctionalEquationsforQuotOrbifoldWreathProducts}


\subsection{Generating Functional Equation for $\Gamma$-Extensions}
\label{subsec-generalTamanoiFunctEqautionWreathProducts}

In this section, we generalize the generating functions of extensions of
multiplicative invariants for wreath symmetric products in
\cite{tamanoi2} to the case of orbifolds that admit a cc-presentation.
In particular, Theorem \ref{th-GeneratingFunctionGammaExtensionsFuncEqnNewTheo}
corresponds to \cite[Proposition 5-4]{tamanoi2}.
For specific choices of $\Gamma$ and $\varphi$,
the formula in Theorem \ref{th-GeneratingFunctionGammaExtensionsFuncEqnNewTheo}
specializes to particularly interesting examples; see
Section \ref{subsec-ExamplesTamanoiFunctEqautionWreathProducts}.

By a \emph{multiplicative orbifold invariant}, we mean a
function $\varphi$ defined on a subclass of Morita equivalence classes of
orbifold groupoids such that
\[
    \varphi(\mathcal{G} \times \mathcal{H}) = \varphi(\mathcal{G})\varphi(\mathcal{H})
\]
where $\mathcal{G} \times \mathcal{H}$ is a product groupoid; see \cite[page 123]{moerdijkmrcun}.
Examples include the (usual) Euler characteristic $\chi$ of the orbit space
and the Euler--Satake characteristic $\chi_{ES}$, see \cite{farsiseaton3}.
We are particularly interested in multiplicative orbifold invariants defined for
all orbifolds that admit cc-presentations.  We restrict
to  the case that $\Gamma$ is a finitely generated discrete group to ensure
that these extensions are finite.

\begin{definition}
\label{def-multInvAssocInvariants}
Let $\varphi$ be a multiplicative orbifold invariant, and
let $\Gamma$ be a finitely generated discrete group.
\begin{itemize}
\item[({\it i})]
The \emph{$\Gamma$-extension $\varphi_\Gamma$ of $\varphi$}
        is defined by
\[
    \varphi_{\Gamma} \left(M \rtimes G \right) =
    \sum_{(\theta) \in HOM(\Gamma, G)/G}
    \varphi \left(M^{\langle\theta\rangle} \rtimes C_G(\theta) \right)
\]
where $(\theta)$ ranges over conjugacy classes of homomorphisms
from $\Gamma$ to $G$ and $\varphi \left(M^{\langle\theta\rangle} \rtimes C_G(\theta) \right)$
is taken to be zero when $M^{\langle\theta\rangle} = \emptyset$.

\item[({\it ii})]
Let $H \leq \Gamma$ be a subgroup of finite index,
and let $(\Gamma/H)$ denote the isomorphism class of $\Gamma/H$ as a $\Gamma$-set.
The $(\Gamma/H)$-extension of a multiplicative orbifold
invariant $\varphi$ is defined by
\[
    \varphi_{(\Gamma / H)} (M\rtimes G)= \sum_{[\rho] \in HOM(H, G)/(N_{\Gamma}(H) \times G)} \varphi
    \left(  M^{\langle \rho\rangle}\rtimes Aut^{P_{\rho}}_{\Gamma,\,G}  \right).
\]
Here, $[\rho]$ ranges over $(N_{\Gamma}(H) \times G)$-orbits of
homomorphisms from $H$ to $G$ where the action on $\rho \in HOM(H, G)$ is that
of Equation \eqref{eq-NormTimesGAction}.
As well, $p_{\rho}\co P_{\rho} = \Gamma
\times_{\rho} G \to \Gamma / H$ is a
$\Gamma$-irreducible $G$-principal bundle, and  $Aut^{P_{\rho}}_{\Gamma,\,G}$ is the
automorphism group of $P_{\rho}$ described in \cite[Theorem 4-4]{tamanoi2} and recalled below.
\end{itemize}
\end{definition}

If $M \rtimes G$ is a cc-presentation of the orbifold $Q$, it follows from
\cite[Theorem 3.5]{farsiseaton2} that
\[
    \bigsqcup\limits_{(\theta) \in HOM(\Gamma, G)/G} M^{\langle\theta\rangle} \rtimes C_G(\theta)
\]
is a presentation of the orbifold of $\Gamma$-sectors of $Q$ defined in
\cite[Definition 2.3]{farsiseaton1}, and hence that $\varphi_\Gamma$ corresponds to the application of
$\varphi$ to the $\Gamma$-sectors.  As the $\Gamma$-sectors of a closed orbifold
consist of a finite disjoint union of closed orbifolds by \cite[Lemma 2.9]{farsiseaton1},
$M^{\langle\theta\rangle} \rtimes C_G(\theta) = \emptyset$ for all but finitely many
elements of $HOM(\Gamma, G)/G$ so that $\varphi_\Gamma$ is finite.  That
$\varphi_\Gamma$ is multiplicative is a consequence of
\cite[Proposition 3.2]{farsiseaton3}.  In particular, when $\varphi$ is equal to
the Euler or Euler--Satake characteristic, the definition of $\chi_\Gamma$
and $\chi_\Gamma^{ES}$ given above coincides with that of \cite{farsiseaton3}.
Similarly, the invariant $\varphi_{(\Gamma / H)} (M\rtimes G)$ can be interpreted geometrically
in terms of $G$-bundles over the covering space
\[
    \tilde{\Sigma} \times_{\Gamma} (\Gamma / H),
\]
where $\Sigma$ is a closed manifold with fundamental group $\Gamma$ and
$\tilde{\Sigma}$ is the universal covering space of $\Sigma$.

Note that by \cite[Theorem 4-4]{tamanoi2},
$Aut^{P_{\rho}}_{\Gamma,\,G}$ is isomorphic to the quotient
$H\backslash T_\rho$ where $T_\rho$ is the isotropy group of $\rho$
in $N_\Gamma(H) \times G$ with respect to the action given in
Equation \eqref{eq-NormTimesGAction}.  Using this identification, the
action of $Aut^{P_{\rho}}_{\Gamma,\,G}$ on $M^{\langle \rho \rangle}$
is given by $H(u,g)x = gx$ as in \cite[Proposition 5-3]{tamanoi2}. In
particular, as $H$ has finite index in $\Gamma$ and $G$ acts locally
freely, the action of each $Aut^{P_{\rho}}_{\Gamma,\,G}$ on
$M^{\langle \rho \rangle}$ is clearly locally free, and hence
presents an orbifold.
As in the case of $G$ finite, when $H=\Gamma$, the $N_{\Gamma}(\Gamma)$-action on
$\rho\co \Gamma \to G$ is absorbed by conjugation by
$G$.  It follows that $HOM(H, G)/ (N_{\Gamma}(H) \times G) = HOM(H, G)/ G$ and
$Aut^{P_{\rho}}_{\Gamma,\,G} = C_G(\rho)$, so that
$\varphi_{(\Gamma/H)} =  \varphi_{\Gamma}$.
That $\varphi_{(\Gamma/H)}$ is in general finite follows from the following.

\begin{proposition}
\label{prop-MultInvarGamSetX}
Let $\Gamma$ be a finitely generated discrete group, and let $M \rtimes G$
be a cc-presentation of the orbifold $Q$.
Let $H \leq \Gamma$ be a subgroup of index $n$.
Then for a multiplicative orbifold invariant $\varphi$, we have
\[
\varphi_{(\Gamma / H)}(M\rtimes G)   = \ \sum_{(\tau) \in \pi^{-1}((\Gamma/H))} \varphi
\left(  (M^n)^{\langle\tau\rangle}\rtimes C_{ G({\mathcal S}_n)}(\tau)  \right),
\]
where
$\pi \co HOM(\Gamma, G({\mathcal S}_n))/ G({\mathcal S}_n ) \to HOM(\Gamma, {\mathcal S}_n)/ {\mathcal S}_n$
denotes composition with the obvious homomorphism $G(\mathcal{S}_n) \to \mathcal{S}_n$ and
$HOM(\Gamma, {\mathcal S}_n)/ {\mathcal S}_n$ is identified with the set of
isomorphism classes of $\Gamma$-sets of order $n$.
\end{proposition}

The proof is identical to \cite[Proposition 6-1]{tamanoi2} and hence omitted.  Recall that
$C_{ G({\mathcal S}_n)}(\tau) \cong Aut_{\Gamma,\,G}^{P_\tau}$ for a $\Gamma$-$G$-principal bundle associated
by Theorem \ref{th-CrossedPdctBundleChr}({\it i}) to the conjugacy class $(\tau)$ of $\tau$; see Equation \eqref{eq-CG=Aut}.

When $M \rtimes G$ is a cc-presentation of the orbifold $Q$,
the \emph{$n$th wreath symmetric product $M^n \rtimes (G(\mathcal{S}_n))$
of $Q$} is the orbifold presented by $M^n \rtimes G(\mathcal{S}_n))$
where $G(\mathcal{S}_n)$ is the wreath product as in Remark \ref{rem-SpecialCaseXisN}
and the action of $((g_1, \ldots, g_n), s) \in G(\mathcal{S}_n)$ on $(x_1, \ldots, x_n) \in M^n$
is given by
\[
    ((g_1, \ldots, g_n),s)  (x_1, \dots, x_n)
    =
    (g_1 x_{s^{-1}(1)}, \dots, g_n x_{s^{-1}(n)}).
\]
The proof of \cite[Proposition 5-4]{tamanoi2} extends directly to this case;
we recall it briefly below.

\begin{theorem}
\label{th-GeneratingFunctionGammaExtensionsFuncEqnNewTheo}
Let $\Gamma$ be a finitely generated discrete group, and let $M \rtimes G$
be a cc-presentation of the orbifold $Q$.  For a multiplicative orbifold
invariant $\varphi$,
\[
    \sum_{n\geq 0} q^n \varphi_{\Gamma} \left(M^n\rtimes G({\mathcal S}_n)\right)
    = \prod_{(H), [ \rho]} \left( \sum_{n \geq 0}\,
    q^{|\Gamma/H| n} \varphi\left((M^{\langle \rho\rangle})^{n}\rtimes
    Aut_{\Gamma,\,G}^{P_\rho}({\mathcal S}_n)\right)\right),
\]
where the product runs over all conjugacy classes $(H)$ of
subgroups of $\Gamma$ of finite index and all elements $[\rho]$ of $HOM(H,G)/(N_{\Gamma}(H)\times G)$,
and $P_\rho$ is the $\Gamma$-$G$-principal bundle corresponding to $\rho$ by
Theorem \ref{th-CrossedPdctBundleChr}({\it ii}).
\end{theorem}

\begin{proof}
A homomorphism $\theta\co\Gamma\to G(\mathcal{S}_n)$ corresponds by
Theorem \ref{th-CrossedPdctBundleChr}({\it i}) to a $\Gamma$-$G$-principal bundle
$P_\theta$ over $\mathbf{n}=\{ 1, 2, \ldots, n \}$.  Such a bundle decomposes into
a finite collection of $\Gamma$-irreducible $G$-principal bundles over
the $\Gamma$-orbits in $\mathbf{n}$, each identified with $\Gamma/H$ for
some $H \leq \Gamma$ with finite index.  Each irreducible
bundle then corresponds to an element of $HOM(H,G)/(N_{\Gamma}(H)\times G)$
by Theorem \ref{th-CrossedPdctBundleChr}({\it ii}).  Let $r(H, \rho)$
denote the number of $\Gamma$-irreducible $G$-principal bundles whose isomorphism class
corresponds to $\rho\co H \to G$ in the decomposition of $P_\theta$.
Note that $\sum_{(H), [\rho]}  | \Gamma /H|\, r(H, \rho)\ = \ n$.

By the results detailed in \cite[Sections 3 and 4]{tamanoi2},
which hold for all groups $G$,
\[
    \varphi_\Gamma(M^n\rtimes G({\mathcal S}_n))
    =
    \ \sum\limits_{r(H, \rho)} \prod\limits_{(H)} \prod\limits_{[\rho]}
    \ \varphi_\Gamma \left( (M^{\langle\rho\rangle})^{r(H, \rho)}\rtimes
    Aut_{\Gamma,\,G}^{P_\rho}({\mathcal S}_{r(H, \rho)}) \right),
\]
where the sum over $r(H, \rho)$ is over all
sets of non-negative integers such that
$\sum_{(H), [\rho]}  | \Gamma /H|\, r(H, \rho)\ = \ n$.
Taking the sum over $n$ and rearranging terms yields
\[
    \sum\limits_{n \geq 0} q^n  \varphi_\Gamma(M^n\rtimes G({\mathcal S}_n))
    =
    \prod\limits_{(H)} \prod\limits_{[\rho]} \sum\limits_{r \geq 0}  q^{| \Gamma /H| r}
    \ \varphi_\Gamma\left( (M^{\langle \rho \rangle})^{r}\rtimes
    Aut_{\Gamma,\,G}^{P_\rho}({\mathcal S}_{r})\right).
\]
\end{proof}


\subsection{The Euler and Euler--Satake Characteristics}
\label{subsec-ExamplesTamanoiFunctEqautionWreathProducts}

In this section, we interpret Theorem \ref{th-GeneratingFunctionGammaExtensionsFuncEqnNewTheo}
for specific orbifold invariants.  In particular, we consider the standard Euler characteristic
$\chi(M\rtimes G) = \chi(M/G)$ and the Euler--Satake characteristic $\chi_{ES}(M\rtimes G)$,
extending \cite[Theorems 5-5 and 6-3]{tamanoi2} to the case of orbifolds that admit a cc-presentation.

We first consider the $\Gamma$-extension $\chi_\Gamma$ of the usual Euler characteristic
$\chi(M/G)$ given in Definition \ref{def-multInvAssocInvariants}.
The following is needed for the case of $\Gamma$ abelian, see \cite[Lemma 6-2]{tamanoi2}.

\begin{lemma}
\label{lem-GenerFunctGammaExtensStandardEulerCharacteristic}
Let $\Gamma$ be a finitely generated abelian discrete group, and let $M \rtimes G$
be a cc-presentation of the orbifold $Q$.
For any subgroup $H \leq \Gamma$ of finite index in  $\Gamma $ we have

\[
\chi_{(\Gamma/H)} (M\rtimes G)\, = \,
\sum_{{(\rho) \in HOM(H, G)/G}}\
\chi_{}  \left( M^{\langle\rho\rangle}\rtimes
C_G(\rho) \right) = \chi_H( M \rtimes G).
\]
\end{lemma}

\begin{proof}
By Definition \ref{def-multInvAssocInvariants}
\[
\chi_{(\Gamma/H)} (M\rtimes G)
= \ \sum_{[\rho] \in HOM(H, G)/(N_{\Gamma}(H) \times G)} \chi
\left(  M^{\langle \rho\rangle}\rtimes Aut^{P_{\rho}}_{\Gamma,\,G}  \right).
\]
As $\Gamma$ is abelian, $HOM(H, G)/(N_{\Gamma}(H) \times G) = HOM(H, G)/ G$,
and
\[
Aut_{\Gamma,\,G}^{P_{\rho}}
\cong  \Gamma \times_{\rho} C_G(\rho),
\]
for $\rho \in HOM(H, G)$ by \cite[Equation 4-4]{tamanoi2}.
Thus we have
\begin{align*}
    \chi_{(\Gamma/H)} (M\rtimes G)
    &=
    \sum\limits_{(\rho) \in HOM(H, G)/G}\,  \chi \left(M^{\langle\rho\rangle}\rtimes
    Aut_{\Gamma,\,G}^{P_{\rho}} \right)
            \\
    &=
    \sum\limits_{(\rho) \in HOM(H, G)/G}\,  \chi \left(M^{\langle\rho\rangle}\rtimes
   ( \Gamma \times_{\rho} C_G(\rho) \right) .
\end{align*}

Hence, if $\{  \gamma_j \} $ is a set of
representatives for the cosets
$H \backslash \Gamma$, we have
\[
    \Gamma \times_{\rho} C_G (\rho) = \coprod_j
    \{       i(\gamma_j) C_G(\rho)  \} .
\]
Recalling that the action of $Aut_{\Gamma,\,G}^{P_{\rho}}$ on $M^{\langle \rho\rangle}$
is given by $H(u,g)x = gx$, the image of the natural injection $i \co \Gamma \to  ( \Gamma \times_{\rho} C_G(\rho) )$
acts trivially on $M^{\langle\rho\rangle}$.  Therefore,
\[
    M^{\langle\rho\rangle} /
   (\Gamma \times_{\rho} C_G(\rho) )
    =
    M^{\langle\rho\rangle} /
    C_G(\rho),
\]
and so
\[
    \chi \left(M^{\langle\rho\rangle}\rtimes
   ( \Gamma \times_{\rho} C_G(\rho) )\right) = \chi \left(M^{\langle\rho\rangle}\rtimes
    C_G(\rho) \right),
\]
from which the result follows.
\end{proof}

With this, we have the following.

\begin{theorem}
[$\Gamma$-Extensions of the Euler Characteristic]
\label{th-GenerFunctGammaExtensStandardEulerCharacteristic}
Let $\Gamma$ be a finitely generated discrete group, and let $M \rtimes G$
be a cc-presentation of the orbifold $Q$.  For the multiplicative orbifold
invariant $\chi$, we have
\begin{equation}
\label{eq-GenFuncEC}
\sum_{n \geq 0}q^n \,  \chi_{\Gamma} (M^n\rtimes G({\mathcal S}_n))
= \prod_{r \geq 1} (1- q^r)^{-\sum_{(H_r)}
\chi_{(\Gamma /H_r)} (M\rtimes G) },
\end{equation}
where $(H_r)$ runs over the $\Gamma$-conjugacy classes of
subgroups of $\Gamma$ of finite index $r$.
If $\Gamma$ is abelian, then
\begin{equation}
\label{eq-AbelGenFuncEC}
\sum_{n\geq 0}q^n \,  \chi_{\Gamma} (M^n\rtimes G({\mathcal S}_n))
= \prod_{r \geq 1} (1- q^r)^{-\sum_{H_r}
\chi_{H_r} (M\rtimes G) },
\end{equation}
where $H_r$ runs over all subgroups of $\Gamma$ of finite index $r$.
\end{theorem}

\begin{proof}
By Theorem \ref{th-GeneratingFunctionGammaExtensionsFuncEqnNewTheo}, we have that
$\sum_{n \geq 0}q^n \,  \chi_{\Gamma} (M^n\rtimes G({\mathcal S}_n))$ is given by
\[
\sum_{n\geq 0}q^n \,  \chi_{\Gamma} (M^n\rtimes G({\mathcal S}_n))
\ =\ \prod_{(H), [\rho]} \ \sum_{r \geq 0}  q^{| \Gamma /H| r}
\ \chi\left( \mathcal (M^{\langle \rho \rangle})^r \rtimes
Aut_{\Gamma,\,G}^{P_\rho} \right),
\]
where the product ranges over $\Gamma$-conjugacy classes $(H)$ of subgroups
$H \leq \Gamma$ of finite index as well as
$(N_\Gamma(H) \times G)$-orbits $[\rho]$ of homomorphisms
$\rho \in HOM(H, G)$.  By MacDonald's formula
\cite[Theorem 5.8]{farsiseaton3}, we have that this is equal to
\[
\prod_{(H), [\rho]}  (1- q^{|\Gamma /H|})^{- \chi\left( M^{\langle \rho \rangle}\rtimes
Aut_{\Gamma,\,G}^{P_\rho}\right)}
=
\prod_{r \geq 1} (1- q^{r})^{- \sum_{(H_r)}  \sum_{[\rho]}
\chi\left( M^{\langle \rho \rangle}\rtimes
Aut_{\Gamma,\,G}^{P_\rho}\right)},
\]
where $(H_r)$ ranges over the
$\Gamma$-conjugacy classes of subgroups of finite index $r$ in
$\Gamma$ and  $[\rho]$ ranges over
$(N_\Gamma(H_r) \times G)$-conjugacy classes of homomorphisms.
Noting that the last summation over $[\rho]$ yields exactly
$\chi_{(\Gamma/H_r)}(M\rtimes G)$, Equation \eqref{eq-GenFuncEC} follows.
Then Equation \eqref{eq-AbelGenFuncEC} follows from Lemma
\ref{lem-GenerFunctGammaExtensStandardEulerCharacteristic}.
\end{proof}

For the Euler--Satake characteristic
$\chi_{ES}(M\rtimes G)$ (see \cite{farsiseaton3}),
define $\chi^{ES}_{\Gamma}(M\rtimes G)$ to be the corresponding
$\Gamma$-extension defined in Definition \ref{def-multInvAssocInvariants}.
Then we have the following.

\begin{theorem}
[$\Gamma$-Extension of the Euler--Satake Characteristic]
\label{th-GenerFunctGammaExtensEulerSatakeCharacteristic}
Let $\Gamma$ be a finitely generated discrete group, and let $M \rtimes G$
be a cc-presentation of the orbifold $Q$.  For the multiplicative orbifold
invariant $\chi_{ES}(M\rtimes G)$, we have
\begin{equation}
\label{eq-GenerFunctGammaExtensEulerSatakeCharacteristic}
\sum_{n \geq 0}q^n \,  \chi^{ES}_{\Gamma} (M^n\rtimes G({\mathcal S}_n))
= \exp \left( \sum_{n \geq 1} \frac{q^n}n
\sum_{H\leq \Gamma : |\Gamma /H| =n}  \chi^{ES}_{H} (M\rtimes G)  \right).
\end{equation}
\end{theorem}

\begin{proof}
We follow the proof of \cite[Theorem 5-5]{tamanoi2}; the main modification is in using
orbifold covers to avoid dealing with infinite orders of $G$ and its subgroups.

By Theorem \ref{th-GeneratingFunctionGammaExtensionsFuncEqnNewTheo}, we have
\begin{equation}
\label{eq-GenerFunctGammaExtensEulerSatakeCharProof}
    \sum\limits_{n \geq 0}q^n \,  \chi^{ES}_{\Gamma} (M^n\rtimes G({\mathcal S}_n))
    =
    \prod\limits_{(H), [ \rho]} \; \sum\limits_{n \geq 0}\,
    q^{|\Gamma/H| n} \chi_{ES}\left((M^{\langle \rho\rangle})^{n}\rtimes
    Aut_{\Gamma,\,G}^{P_\rho}({\mathcal S}_n)\right).
\end{equation}

By \cite[Theorem 4-2]{tamanoi2}, the index $[Aut_{\Gamma,\,G}^{P_\rho}:C_G(\rho)]$ is
given by $|N_\Gamma^\rho(H)/H|$, which is finite, so that
\[
    [Aut_{\Gamma,\,G}^{P_\rho}({\mathcal S}_n):C_G(\rho)({\mathcal S}_n)]
    =
    |N_\Gamma^\rho(H)/H|^n.
\]
It follows that
\[
    (M^{\langle \rho\rangle})^{n}\rtimes C_G(\rho)({\mathcal S}_n)
    \longrightarrow
    (M^{\langle \rho\rangle})^{n}\rtimes Aut_{\Gamma,\,G}^{P_\rho}({\mathcal S}_n)
\]
is an orbifold cover with $|N_\Gamma^\rho(H)/H|^n$ sheets, so that by
\cite[Proposition 13.3.4]{thurston} and \cite[Lemma 2.2, Theorems 2.3 and 5.11]{farsiseaton3}, we have
\begin{align*}
    &\prod\limits_{(H), [ \rho]} \; \sum\limits_{n \geq 0}\,
    q^{|\Gamma/H| n} \;\; \chi_{ES}\left((M^{\langle \rho\rangle})^{n}\rtimes
    Aut_{\Gamma,\,G}^{P_\rho}({\mathcal S}_n)\right)
    \\&\quad=
    \prod\limits_{(H), [ \rho]} \; \sum\limits_{n \geq 0}\,
    \frac{q^{|\Gamma/H| n}}{|N_\Gamma^\rho(H)/H|^n n!}
    \;\; \chi_{ES}\left(M^{\langle \rho\rangle}\rtimes C_G(\rho)\right)^n
    \\&\quad=
    \prod\limits_{(H), [ \rho]} \exp \left(
    \frac{q^{|\Gamma/H|}}{|N_\Gamma^\rho(H)/H|}
    \;\; \chi_{ES}\left(M^{\langle \rho\rangle}\rtimes C_G(\rho)\right) \right)
    \\&\quad=
    \exp \left(\sum\limits_{n \geq 1} q^n \sum\limits_{(H):|\Gamma/H|=n}\sum\limits_{[\rho]}
    \frac{|\Gamma/N_\Gamma^\rho(H)|\;\;\chi_{ES}\left(M^{\langle \rho\rangle}\rtimes C_G(\rho)\right)}
    {|\Gamma/N_\Gamma(H)||N_\Gamma(H)/N_\Gamma^\rho(H)||N_\Gamma^\rho(H)/H|}\right)
    \\&\quad=
    \exp \left(\sum\limits_{n \geq 1} q^n \sum\limits_{(H):|\Gamma/H|=n}
    \frac{1}{|N_\Gamma(H)/H|}
    \sum\limits_{(\rho)}
    \chi_{ES}\left(M^{\langle \rho\rangle}\rtimes C_G(\rho)\right)\right).
\end{align*}
In the last equation, note that we switch from summing over $N_\Gamma(H) \times G$-conjugacy classes $[\rho]$
of $\rho\co H \to G$ to $G$-conjugacy classes ${{(\rho) \in HOM(H, G)/G}}$, and the $N_\Gamma(H)$-orbit of $\rho$ has
$|N_\Gamma(H)/N_\Gamma^\rho(H)|$ elements.  Then as the final sum in the last expression is the definition of
$\chi_H^{ES}\left(M^{\langle \rho\rangle}\rtimes C_G(\rho)\right)$, and
each conjugacy class $(H)$ contains $|\Gamma/N_\Gamma(H)|$ elements, this is equal to
\begin{align*}
    &\exp \left(\sum\limits_{n \geq 1} q^n \sum\limits_{H:|\Gamma/H|=n}
    \frac{1}{|\Gamma/N_\Gamma(H)||N_\Gamma(H)/H|}
    \chi_H^{ES}(M\rtimes G) \right)
    \\&\quad=
    \exp \left(\sum\limits_{n \geq 1} \frac{q^n}{n} \sum\limits_{H:|\Gamma/H|=n}
    \chi_H^{ES}(M\rtimes G) \right).
\end{align*}
\end{proof}


\subsection{Examples: Weighted Projective Spaces}
\label{subsec-wrExamplesExpansionExamplesIIItheHinertia}

Examples of orbifolds that are not global quotients and are most naturally described as
cc-presentations are weighted projective spaces, see e.g. \cite{bfnr,jiang}.

\begin{definition}
\label{def-weightedProjSp}
Let $\Sp^{2k-1} \subset \C^k$ denote the $(2k-1)$-sphere with $k \geq 2$, and let $\T^1 \subset \C$ denote the $1$-dimensional
torus.  For a fixed \emph{weight vector} $(n_1, \ldots, n_k)$, a $k$-tuple of positive integers,
define a $\T^1$-action on $\Sp^{2k-1}$ as
\[
    t (z_1, ..., z_k) = (t^{n_1}z_1, \ldots , t^{n_k} z_k),
    \quad\quad
    t\in \T^1, \quad (z_1, \ldots, z_k) \in \Sp^{2k-1}.
\]
It is easy to see that this action is locally free.  The \emph{$(n_1, \ldots, n_k)$-weighted projective space
$\Pb(n_1, \ldots, n_k)$} is the orbifold presented by $\Sp^{2k-1}\rtimes \T^1$.  When the weights $n_i$
are pairwise relatively prime and not all equal to $1$, we call $\Pb(n_1, \ldots, n_k)$ a \emph{teardrop}.
\end{definition}

Note that the standard $\Z_n$-teardrop, an orbifold homeomorphic to $\Sp^2$ with one singular point modeled on
$\R^2\rtimes \Z_n$, $n > 1$, with $\Z_n$ acting as rotations, is given by $\Pb(1, n)$.

In this section, we will illustrate Theorems \ref{th-GenerFunctGammaExtensStandardEulerCharacteristic} and
\ref{th-GenerFunctGammaExtensEulerSatakeCharacteristic} with the weighted projective spaces $\Pb(m, mn, n)$ where $m,n > 1$ are relatively prime
and $\Gamma$ is $\Z$, $\Z^{\ell}$, or the fundamental group of a
closed orientable surface of genus greater than 1.
Let $M = \Sp^5$, and note that a point $(z_1, z_2, z_3) \in \Sp^5$ is in a singular orbit if and only if $z_1 = 0$ or $z_3 = 0$.
Hence, the singular set of $\Pb(m, mn, n)$ consists of the union of $\Pb(mn,n)$, a standard $\Z_m$-teardrop with trivial $\Z_n$-action
corresponding points with $z_1 = 0$, and $\Pb(m,mn)$, a standard $\Z_n$-teardrop with trivial $\Z_m$-action corresponding to points
with $z_3 = 0$; the intersection of these sets is the orbit of $(0,1,0)$ which has isotropy $\Z_{mn}$.
In particular, as standard teardrops appear as sectors of $\Pb(m, mn, n)$ and are not global quotients,
$\Pb(m, mn, n)$ is itself not a global quotient.  See \cite{jiang} for a complete description of the $\Z$-sectors
(i.e. inertia orbifold) of weighted projective spaces.

By \cite[Equation 1.2]{bfnr}, composing the maps $z_1 \mapsto z_1^n$ and $z_3 \mapsto z_3^m$ induces a homeomorphism between
$\Pb(m, mn, n)$ and the standard projective space $\Pb(1,1,1) = \Pb^2$.  Using this homeomorphism, we have
$\chi\big(\Pb(m, mn, n)\big) = \chi(\Pb^2) = 3$.  Similarly, considering a simplicial decomposition subordinate to the singular set,
removing the union of two spheres that intersect at a point, and replacing
them with the singular set of $\Pb(m, mn, n)$, one computes that
\[
    \chi_{ES}\big(\Pb(m, mn, n)\big) = \frac{1}{n}\left(1 + \frac{1}{m}\right)
        + \frac{1}{m}\left(1 + \frac{1}{n}\right) - \frac{1}{mn} = \frac{1 + m + n}{mn}.
\]


\begin{example}[$\Pb(m,mn,n)$ with $\Gamma = \Z$]
\label{ex-WeightProjPZ}
To evaluate the expressions in Theorems \ref{th-GenerFunctGammaExtensStandardEulerCharacteristic}  and
\ref{th-GenerFunctGammaExtensEulerSatakeCharacteristic}, we will need to compute
$\chi_{H_r}(M \rtimes \T^1)$ and $\chi_{H_r}^{ES}(M \rtimes \T^1)$ where $H_r = \langle r \rangle \leq \Z$
denotes the unique subgroup $\Z$ of index $r$.  Noting that $H_r$ is isomorphic to $\Z$, it is sufficient to
consider $\chi_{\Z}(M \rtimes \T^1)$ and $\chi_{\Z}^{ES}(M \rtimes \T^1)$, i.e. the Euler and Euler--Satake
characteristics, respectively, of
\[
    \bigsqcup\limits_{(\theta)\in HOM(\Z,\T^1)/\T^1} M^{\langle\theta\rangle}\rtimes C_{\T^1}(\theta)
    =
    \bigsqcup\limits_{\theta\in HOM(\Z,\T^1)} M^{\langle\theta\rangle}\rtimes \T^1,
\]
where recall that $M = \Sp^5$.  Choose a generator $\zeta$ of $\Z_{mn} \leq \T^1$ and
note that $M^{\langle\theta\rangle} = \emptyset$ unless the image $\Imt(\theta)$ of $\theta$
is a subgroup of $\Z_{mn}$.  Moreover,
letting $M_1 = \{ (0, z_2, z_3) \} \subset M$, $M_3 = \{ (z_1, z_2, 0) \} \subset M$,
and $M_{13} = \{ (0, z_2, 0) \} \subset M$, one checks that for nontrivial $\theta$,
\begin{itemize}
\item   $M^{\langle\theta\rangle} = M_1$ if $\Imt(\theta) \leq \langle\zeta^m\rangle \cong \Z_n$, so
        $M^{\langle\theta\rangle} \rtimes \T^1$ presents $\Pb[mn, n]$;
\item   $M^{\langle\theta\rangle} = M_3$ if $\Imt(\theta) \leq \langle\zeta^n\rangle \cong \Z_m$, so
        $M^{\langle\theta\rangle} \rtimes \T^1$ presents $\Pb[m, mn]$; and
\item   $M^{\langle\theta\rangle} = M_{13}$ if $\Imt(\theta)$ is not contained in $\langle \zeta^n\rangle$ or $\langle \zeta^m\rangle$,
        so $M^{\langle\theta\rangle} \rtimes \T^1$ presents $\Pb[mn]$, a point with trivial $\Z_{mn}$-action.
\end{itemize}
By considering the image of a generator of $\Z$, one computes that of the $mn-1$ nontrivial elements of
$HOM(\Z,\Z_{mn})$, $m-1$ have image in $\langle\zeta^n\rangle$, $n-1$ have image in $\langle\zeta^m\rangle$,
with $(m-1)(n-1)$ remaining.  It follows that $\bigsqcup_{\theta\in HOM(\Z,\T^1)} M^{\langle\theta\rangle}\rtimes \T^1$
consists of
\begin{itemize}
\item   $\Pb(m,mn,n)$ corresponding to the trivial homomorphism with
        \[
            \chi\big(\Pb(m,mn,n)\big) = 3
            \quad\quad\mbox{and}\quad\quad
            \chi_{ES}\big(\Pb(m,mn,n)\big) = \frac{1 + m + n}{mn};
        \]
\item   $n-1$ copies of $\Pb(mn,n)$ with
        \[
            \chi\big(\Pb(mn,n)\big) = 2
            \quad\quad\mbox{and}\quad\quad
            \chi_{ES}\big(\Pb(mn,n)\big) = \frac{1}{n}\left(1 + \frac{1}{m}\right);
        \]
\item   $m-1$ copies of $\Pb(m,mn)$ with
        \[
            \chi\big(\Pb(m,mn)\big) = 2
            \quad\quad\mbox{and}\quad\quad
            \chi_{ES}\big(\Pb(m,mn)\big) = \frac{1}{m}\left(1 + \frac{1}{n}\right);
            \quad\mbox{and}
        \]
\item   $(m-1)(n-1)$ copies of $\Pb(mn)$ with
        \[
            \chi\big(\Pb(mn)\big) = 1
            \quad\quad\mbox{and}\quad\quad
            \chi_{ES}\big(\Pb(mn)\big) = \frac{1}{mn}.
        \]
\end{itemize}
Hence
\[
    \chi_{\Z}(M \rtimes \T^1)
    =
    3 + 2(n-1) + 2(m-1) + (m-1)(n-1)
    =
    m + n + mn
\]
and
\begin{align*}
    &\chi_{\Z}^{ES}(M \rtimes \T^1)
    \\&\quad =
    \frac{1 + m + n}{mn}
    +
    \frac{n-1}{n}\left(1 + \frac{1}{m}\right)
    +
    \frac{m-1}{m}\left(1 + \frac{1}{n}\right)
    +
    \frac{(m-1)(n-1)}{mn}
    =
    3.
\end{align*}
Note that this also follows from the fact that $\chi_\Z^{ES}(Q) = \chi(Q)$ for an arbitrary closed orbifold $Q$;
see \cite{farsiseaton3}.

With this, we have by Theorem \ref{th-GenerFunctGammaExtensStandardEulerCharacteristic}, Equation \eqref{eq-AbelGenFuncEC} that
\[
    \sum_{n\geq 0}q^n \,  \chi_\Z (M^n\rtimes \T^1({\mathcal S}_n))
    = \prod_{r \geq 1} (1- q^r)^{-(m + n + mn)},
\]
recovering \cite[Equation (18)]{farsiseaton3} in this case.  Similarly,
by Theorem \ref{th-GenerFunctGammaExtensEulerSatakeCharacteristic},
\[
    \sum_{n \geq 0}q^n \,  \chi^{ES}_\Z (M^n\rtimes \T^1({\mathcal S}_n))
    =
    \sum_{n \geq 0}q^n \,  \chi (M^n\rtimes \T^1({\mathcal S}_n))
    = \exp \sum_{n \geq 1} \frac{3q^n}{n}
    = (1 - q)^{-3},
\]
also recovering \cite[Equation (18)]{farsiseaton3}.
\end{example}


\begin{example}[$\Pb(m,mn,n)$ with $\Gamma = \Z^\ell$]
\label{ex-WeightProjPZell}
Let $H_r \leq \Z^\ell$ be a subgroup of finite index $r$, and let $h_1, \ldots, h_\ell$ denote a choice
of generators for $\Z^\ell$.  Expressing $H_r$ in Smith normal form, we have that
\[
    H_r = \langle h_1^{\alpha_1}, \ldots, h_\ell^{\alpha_\ell} \rangle
\]
for an ordered $\ell$-tuple $\alpha = (\alpha_1, \ldots, \alpha_\ell)$ of positive integers.
Then $r = \prod_{j=1}^\ell \alpha_j$ and $H_r \cong \Z^\ell$.  We let
$\mathcal{W}_r$ denote the set of $\ell$-tuples $\alpha$ such that
$\prod_{j=1}^\ell \alpha_j = r$; explicit formulas for the cardinality $|\mathcal{W}_r|$
are given in \cite[Lemma 5-6]{tamanoi2}.

Counting the images of generators as in Example \ref{ex-WeightProjPZ}, there are
$n^\ell - 1$ nontrivial $\theta\co\Z^\ell\to\Z_{mn}$ with image contained in
$\langle\zeta^m\rangle \cong \Z_n$, $m^\ell - 1$ nontrivial $\theta\co\Z^\ell\to\Z_{mn}$
with image contained in $\langle\zeta^n\rangle \cong \Z_m$, and
$(mn)^\ell - (m^\ell - 1) - (n^\ell - 1) - 1 = (m^\ell - 1)(n^\ell - 1)$
nontrivial $\theta\co\Z^\ell\to\Z_{mn}$ with image contained in neither
$\langle\zeta^m\rangle$ nor $\langle\zeta^n\rangle$.  As $M^{\langle\theta\rangle}\ltimes\T^1$
depends only on the image $\Imt(\theta)$, we have following the computations in
Example \ref{ex-WeightProjPZ} that
\begin{align*}
    \chi_{\Z^\ell}(M \rtimes \T^1)
    =   3 + 2(n^\ell - 1) + 2(m^\ell - 1) + (m^\ell - 1)(n^\ell - 1)
    =   m^\ell + n^\ell + (mn)^\ell,
\end{align*}
and
\begin{align*}
    &\chi_{\Z^\ell}^{ES}(M \rtimes \T^1)
    \\&\quad =
    \frac{1 + m + n}{mn}
    +
    \frac{n^\ell-1}{n}\left(1 + \frac{1}{m}\right)
    +
    \frac{m^\ell-1}{m}\left(1 + \frac{1}{n}\right)
    +
    \frac{(m^\ell-1)(n^\ell-1)}{mn}
    \\&\quad =
    m^{\ell-1} + n^{\ell-1} + (mn)^{\ell-1}.
\end{align*}
By Theorem \ref{th-GenerFunctGammaExtensStandardEulerCharacteristic}, Equation \eqref{eq-AbelGenFuncEC},
\[
    \sum_{n\geq 0}q^n \,  \chi_{\Z^\ell} (M^n\rtimes \T^1({\mathcal S}_n))
    = \prod_{r \geq 1} (1- q^r)^{-|\mathcal{W}_r| \big(m^\ell + n^\ell + (mn)^\ell\big)}
\]
and by Theorem \ref{th-GenerFunctGammaExtensEulerSatakeCharacteristic},
\[
    \sum_{n \geq 0}q^n \,  \chi^{ES}_{\Z^\ell} (M^n\rtimes \T^1({\mathcal S}_n))
    = \exp \sum_{n \geq 1} \frac{q^n}{n}
    |\mathcal{W}_r| \big(m^{\ell-1} + n^{\ell-1} + (mn)^{\ell-1}\big).
\]
Again, compare \cite[Equation (18)]{farsiseaton3}
\end{example}


\begin{example}[$\Pb(m,mn,n)$ with $\Gamma =\pi_1(\Sigma_{g+1})$]
\label{ex-WeightProjPZellGammaSurface}
Let $\Gamma = \pi_1(\Sigma_{g+1})$ be the fundamental group of a closed oriented
surface $\Sigma_{g+1}$ of genus $(g+1) \geq 1$, and note that when $g = 0$, $\pi_1(\Sigma_1) = \Z^2$.
For $g + 1 > 1$, this case is similar to that of $\Z^\ell$,
because every subgroup of $\Gamma$ of finite index is as well the fundamental group of
a closed oriented surface; see \cite[page 829]{tamanoi2}.
Specifically, a subgroup of index $r$ of $\pi_1(\Sigma_{g+1})$ is isomorphic to $\pi_1(\Sigma_{gr+1})$.
Let $j_{r}(\pi_1(\Sigma_{g+1}))$ denote the number of subgroups of $\pi_1(\Sigma_{g+1})$ of index
$r$, see \cite[page 830]{tamanoi2} for a discussion of calculability of $j_{r}(\pi_1(\Sigma_{g+1}))$
when $g+1 > 1$. Noting that $\pi_1(\Sigma_{gr+1})$ has $2({gr+1})$ generators, Theorem
\ref{th-GenerFunctGammaExtensStandardEulerCharacteristic}, Equation \eqref{eq-AbelGenFuncEC}
yields
\begin{align*}
    &\sum_{n\geq 0}q^n \,  \chi_{ \pi_1(\Sigma_{g+1})} (M^n\rtimes \T^1({\mathcal S}_n))
    \\&\quad\quad=
    \prod_{r \geq 1} (1- q^r)^{-j_{r}(\pi_1(\Sigma_{g+1}))  \big(m^{2({gr+1})}+ n^{2({gr+1})} + (mn)^{2({gr+1})} \big)},
\end{align*}
and Theorem \ref{th-GenerFunctGammaExtensEulerSatakeCharacteristic} yields
\small
\begin{align*}
    &\sum_{n \geq 0}q^n \,  \chi^{ES}_{\pi_1(\Sigma_{g+1})} (M^n\rtimes \T^1({\mathcal S}_n))
    \\&\quad\quad=
    \exp \sum_{n \geq 1} \frac{q^n}{n}
    j_{r}(\pi_1(\Sigma_{g+1}))  \big(m^{2 gr +1} + n^{2 gr +1} + (mn)^{2 gr +1}\big).
\end{align*}
\normalsize
\end{example}

\begin{remark}
\label{rem-OtherGamma}
Using the methods of Examples \ref{ex-WeightProjPZ} and \ref{ex-WeightProjPZell}, one may produce
explicit formulas for the extensions of the Euler and Euler--Satake characteristics of $\Pb(m, mn, n)$
and other weighted projective spaces associated to other finitely generate groups $\Gamma$.
For example, if $\Gamma\cong \Z^{\ell} \oplus \Z_{r_1}\oplus...\oplus  \Z_{r_s}$ is a
finitely generated abelian group, one can in principle compute Equations \eqref{eq-AbelGenFuncEC}
and \eqref{eq-GenerFunctGammaExtensEulerSatakeCharacteristic} using the above methods,
though some challenging bookkeeping involving the prime decompositions
of $m,n$ , and $r_j$, $j=1,...,s$,  might arise in keeping track
of homomorphisms from $\Gamma$ to $\Z_{nm}$.
Note that as $\T^1$ is abelian, the extensions associated to the free group with $\ell$ generators
coincide with those associated to $\Z^\ell$.
\end{remark}


\subsection{Extensions Associated to General $\Gamma$-Sets}
\label{subsec-AbstractTamanoiFunctEqautionWreathProducts}

In this section, we generalize  Definition
\ref{def-multInvAssocInvariants} to include extensions of
multiplicative orbifold invariants associated to arbitrary finite
$\Gamma$-sets where $\Gamma$ is a finitely generated discrete group.
This extends the definition given by \cite[Equation 6-13]{tamanoi2}.

\begin{definition}
\label{def-multInvarAssocGamIsomorClass}
Let $X$ be a finite $\Gamma$-set of order $n$ and
$\varphi$ a multiplicative orbifold invariant.  The extension
of $\varphi$ associated to the
$\Gamma$-isomorphism class $[X]$ of $X$ is defined by
\begin{align*}
    \varphi_{[X]} (M\rtimes G)
    &=     \ \sum_{[P \to X]} \varphi
            \left({\mathcal S}[P_{\rho}\times_G M]^{\Gamma}  \rtimes Aut_{\Gamma,\,G}^{P} \right)  \\
    &=     \sum_{[\theta] \in \pi^{-1}[X]} \varphi
            \left( ( M^n)^{\langle\theta\rangle}\rtimes C_{ G({\mathcal S}_n)}(\theta)  \right),
\end{align*}
where the first sum is over all isomorphism classes of
$\Gamma$-$G$-principal bundles over $X$, $\pi$ is as in Proposition \ref{prop-MultInvarGamSetX},
and ${\mathcal S}[P_{\rho}\times_G M]^{\Gamma}$ denotes the $\Gamma$-invariant sections.
\end{definition}

As in the case of $G$ finite, Theorem \ref{th-CrossedPdctBundleChr}({\it i}) implies
\begin{equation}
\label{eq-GenGammaSettoGammaExt}
\sum_{[X]}\, q^{|X|}\,  \varphi_{[X]} (M\rtimes G)  = \ \sum_{n \geq 0} \, q^n \, \varphi_{\Gamma}
\left( M^n\rtimes G({\mathcal S}_n)  \right),
\end{equation}
where  the first summation is over all isomorphism classes of
$\Gamma$-$G$-principal bundles over finite $\Gamma$-sets $X$.
For a finite $\Gamma$-set $X$, let
$X = \coprod_{(H)} r(H) \Gamma/H$  be
its decomposition into $\Gamma$-orbits where $(H)$ ranges over conjugacy classes
of isotropy groups, $r(H)$ is the number of $\Gamma$-orbits which
are isomorphic to $\Gamma/H$, and $r(H) \Gamma/H$ denotes the disjoint union
of these $r(H)$ isomorphic $\Gamma$-orbits. Then for a multiplicative orbifold
invariant $\varphi$, we have
\[
    \varphi_{[X]} (M\rtimes G) \ = \ \prod_{(H)} \varphi_{r(H)(\Gamma/ H)} (M\rtimes G).
\]
Combining this with Equation \eqref{eq-GenGammaSettoGammaExt} yields the following
interpretation of Theorem \ref{th-GeneratingFunctionGammaExtensionsFuncEqnNewTheo}
in terms of $\Gamma$-sets, which coincides with
\cite[Proposition 6-9]{tamanoi2} for the case of $G$ finite and follows its proof.

\begin{theorem}
\label{decompositionGenerFunction}
Let $\Gamma$ be a finitely generated discrete group, and let $M \rtimes G$
be a cc-presentation of the orbifold $Q$.
Let $\varphi$ be a multiplicative orbifold
invariant and let $X$ be a finite $\Gamma$-set.  With the
notation as above,
\begin{equation}
\label{eq-generatFunctDecomp}
\sum_{n \geq 0} \, q^n \, \varphi_{\Gamma}
\left( M^n\rtimes G({\mathcal S}_n)  \right) = \ \prod_{(H)} \left( \sum_{r \geq 0}q^{r\,|\Gamma/ H| }
\varphi_{r(\Gamma/ H)} (M\rtimes G)\right).
\end{equation}
The generating function of $\varphi_{r(\Gamma/H)}$  is given by
\begin{equation}
\label{eq-generatFunctDecompExplicit}
\sum_{r \geq 0} \, q^r \, \varphi_{r\, (\Gamma/H)}
\left( M\rtimes G  \right)
= \
\prod_{[\rho]}
\sum_{r \geq 0}q^{r }
\varphi\left(( M^{\langle\rho\rangle})^r\rtimes Aut_{\Gamma,\,G}^{P_{\rho}}(\mathcal{S}_r)\right),
\end{equation}
where the product over $[\rho]$ again ranges over $HOM(H,G)/(N_{\Gamma}(H) \times  G)$.
\end{theorem}

\begin{proof}
Letting $P_H$ denote the restriction of a $\Gamma$-$G$-principal bundle $P \to X$ to
$r(H)(\Gamma/H)$, Definition \ref{def-multInvarAssocGamIsomorClass} becomes
\begin{align*}
    \varphi_{[X]} (M\rtimes G)
    &=
    \sum\limits_{[P \to X]} \prod\limits_{(H)} \varphi
    \left({\mathcal S}[P_{H}\times_G M]^{\Gamma}  \rtimes Aut_{\Gamma,\,G}^{P_H} \right)
                \\
    &=
    \prod\limits_{(H)}  \sum\limits_{[P_H]} \varphi
    \left({\mathcal S}[P_{H}\times_G M]^{\Gamma}  \rtimes
    Aut_{\Gamma,\,G}^{P_H} \right)
                \\
    &=
    \prod\limits_{(H)}  \varphi_{r(H) (\Gamma/H)}
    \left( M \rtimes G \right),
\end{align*}
where the sum over $[P_H]$ ranges over
the set of all isomorphism classes
of $\Gamma$-$G$-principal bundles over the
$\Gamma$-set $r(H)(\Gamma/H)$.
Equation \eqref{eq-generatFunctDecomp} follows.

The proof of Equation \eqref{eq-generatFunctDecompExplicit} is analogous to the proof of Theorem
\ref{th-GeneratingFunctionGammaExtensionsFuncEqnNewTheo},
and continues to follow \cite[Proposition 6-9]{tamanoi2}.
For a $\Gamma$-$G$-bundle over the disjoint union of $r$ copies of $\Gamma/H$,
let $r_{\rho}$ be the number of irreducible $\Gamma$-$G$ bundles $P_{\rho}\to \Gamma/H$  appearing in the irreducible
decomposition of $P$, where $[\rho]$  ranges over $HOM(H,G)/(N_{\Gamma}(H) \times  G)$. Then
\begin{align*}
    \sum_{r \geq 0} \, q^r \, \varphi_{r\, (\Gamma/H)}
    \left( M\rtimes G  \right)
    &=
    \sum_{r \geq 0}q^{r }
    \sum_{\sum_{[\rho]}r_{\rho} = r}\prod_{[\rho]}
    \varphi\left(( M^{\langle {\rho}\rangle})^ {r_{\rho}}\rtimes Aut_{\Gamma,\,G}^{P_{\rho}}(\mathcal{S}_{r_{\rho}})\right)
    \\&=
    \sum_{r_{\rho} \geq 0}\prod_{[\rho]}\, q^{r_{\rho} }
    \varphi\left(( M^{\langle {\rho}\rangle})^ {r_{\rho}}\rtimes Aut_{\Gamma,\,G}^{P_{\rho}}(\mathcal{S}_{r_{\rho}})\right)
    \\&=
    \prod_{[{\rho}]}\, \left[ \sum_{r_{\rho} \geq 0}\,  q^{r_{\rho} }
    \varphi\left(( M^{\langle {\rho}\rangle})^ {r_{\rho}}\rtimes Aut_{\Gamma,\,G}^{P_{\rho}}(\mathcal{S}_{r_{\rho}})\right) \right].
    \qedhere
\end{align*}
\end{proof}


\subsection{The $H$-Inertia}
\label{subsec-wrlimynewIIItheHinertia}

In this section, we give a generalized sector construction interpreting
the extensions of multiplicative orbifold invariants associated to transitive $\Gamma$-sets
$\Gamma/H$ as evaluation on sectors.

Consider the action of $N_{\Gamma}(H) \times G$ on $HOM(H, G)$ defined in
Equation \eqref{eq-NormTimesGAction}.  As noted in Subsection
\ref{subsec-generalTamanoiFunctEqautionWreathProducts}, if $T_\rho$
denotes the stabilizer of $\rho \in HOM(H, G)$ in $N_\Gamma(H)
\times G$, then $Aut^{P_{\rho}}_{\Gamma,\,G}$ is isomorphic to $H
\backslash T_{\rho}$, and the action of $Aut^{P_{\rho}}_{\Gamma,\,G}$ on $M^{\langle \rho \rangle}$ depends only on the $G$-factor.
Let
\[
\mathcal{A} = \coprod_{[\rho] \in HOM(H, G)/(N_{\Gamma}(H) \times
G)} [ M^{\langle\rho\rangle} \rtimes  Aut^{P_{\rho}}_{\Gamma,\,G} ].
\]
Then as the action of each $Aut^{P_{\rho}}_{\Gamma,\,G}$ on $M^{\langle\rho\rangle}$
is locally free (see Definition \ref{def-multInvAssocInvariants} and the comments following it),
$\mathcal{A}$ is an orbifold groupoid, and
$\varphi_{(\Gamma/H)}$ is the application of $\varphi$ to the
groupoid $\mathcal{A}$.

As noted in the proof of Lemma
\ref{lem-GenerFunctGammaExtensStandardEulerCharacteristic}, when $\Gamma$ is abelian,
$N_{\Gamma}(H)= \Gamma$ and $T_{\rho} = \Gamma \times_\rho C_G(\rho)$,
so that $Aut^{P_{\rho}}_{\Gamma,\,G} = H \backslash (\Gamma \times_\rho C_G(\rho))$
and $HOM(H, G)/(N_{\Gamma}(H) \times G)= HOM(H, G)/G$.
Therefore the associated groupoid $\mathcal A$ reduces to the product $(M\rtimes G)^\Gamma \times H \backslash
\Gamma$, where $(M\rtimes G)^\Gamma$ is the groupoid of
$\Gamma$-sectors. When in addition $H = \Gamma$,
$Aut^{P_{\rho}}_{\Gamma,\,G} = C_G(\rho)$ so that $\mathcal A$
reduces to the groupoid of $\Gamma$-sectors.

In general, if  $H = \Gamma$ we claim that each connected component
of $\mathcal{A}$ is isomorphic to a $\Gamma$-sector of
$M \rtimes G$, possibly with different multiplicities. In this
case we have $N_{\Gamma}(H)= \Gamma$, $T_{\rho} = \Gamma
\times_{\rho} C_G(\rho)$, and $Aut^{P_{\rho}}_{\Gamma,\,G} =
C_G(\rho)$. Then
\[
    \mathcal{A}
    =
    \coprod_{[\rho] \in HOM(H, G)/(N_{\Gamma}(H) \times G)}
    [ M^{\langle\rho\rangle} \rtimes  C_G(\rho) ],
\]
and hence that each $[\rho]$ corresponds to a union of $\Gamma$-sectors; see
\cite{farsiseaton2}.  As
\[
    (M\rtimes G)^\Gamma
    =
    (M\rtimes G)^H
    =
    \coprod_{(\rho) \in HOM(H, G)/G}
    [ M^{\langle\rho\rangle} \rtimes  C_G(\rho) ]
\]
presents the orbifold of $\Gamma$-sectors,
we see that $\mathcal{A}$ simply identifies $\Gamma$-sectors that are isomorphic via
an element of $N_{\Gamma}(H)$.


\section{Decomposable Functors and Wreath Products}
\label{sec-DecomposableFunctorsAndWreathProducts}

In this section, we assume the orbifold $Q$ is a global quotient orbifold of the form
$M \rtimes G$ so that $G$ is a \emph{finite} group.
We use a modification of a formal functorial functional equation
of Dress and M\"uller \cite{dressmuller} for decomposable functors to determine
a relationship between $\chi_\Gamma^{ES}$ and $\chi_{(\Gamma/H)}^{ES}$.
Note that we modify their approach to replace counting functions for finite sets with
functions related to the invariant $\chi_{ES}$ as defined below.
We follow the notation of \cite[Section 1]{dressmuller}.
Note that the results of this section also follow from \cite{dressmullerII} choosing the Euler--Satake
characteristic as a weight, though the authors were not aware of this when this paper was first prepared.

Fix a global quotient orbifold $M \rtimes G$ and a finitely generated discrete group $\Gamma$.
By Theorem \ref{th-CrossedPdctBundleChr}({\it i}), there is a bijection between the
isomorphism classes of $\Gamma$-$G$-principal bundles over $\Gamma$-sets of
order $n$ and the conjugacy classes of homomorphisms into the wreath
product $G(\mathcal{S}_n)$, i.e. elements of
$HOM(\Gamma, G(\mathcal{S}_n))/G(\mathcal{S}_n)$.  Given a $\Gamma$-$G$-principal bundle
$P$, we let $(\theta)_P \in HOM(\Gamma, G(\mathcal{S}_n))/G(\mathcal{S}_n)$ denote the corresponding conjugacy class.
We will sometimes abuse notation and refer to a homomorphism $\theta_P$, which is defined only up to conjugation.

\begin{definition}
\label{def-functorGGbundles}
Let $\mbox{Ens}$ denote the category with finite sets as objects and
bijective mappings as morphisms, and let $\widetilde{\mbox{Ens}}$ denote the category
with finite sets as objects and injective mappings as morphisms.  For a fixed finitely
generated discrete group $\Gamma$ and finite group $G$, define the covariant functor
\[
    \mathcal{F}_{\Gamma,G} \co \mbox{Ens} \longrightarrow \mbox{Ens}
\]
by assigning to the finite set $\Omega$  (which is given the discrete topology)
the finite set of $\Gamma$-$G$-bundles with total space $\Omega \times G$.  We adopt the convention that
$\mathcal{F}_{\Gamma, G}(\emptyset)$ consists of a single ``empty bundle" corresponding
to the trivial homomorphism from $\Gamma$ into the trivial group.
\end{definition}

Consider the functors
\[
    \mathcal{F}_{\Gamma, G} \times \mathcal{F}_{\Gamma, G} \co \mbox{Ens}^2
    \stackrel{\mathcal{F}_{\Gamma, G}^2}{\longrightarrow} \mbox{Ens}^2
    \stackrel{\times}{\longrightarrow} \mbox{Ens} \stackrel{\imath}{\longrightarrow}
    \widetilde{\mbox{Ens}}
\]
and
\[
   \mathcal{F}_{\Gamma, G} \times \sqcup \co
   \mbox{Ens}^2 \stackrel{\sqcup}{\longrightarrow} \mbox{Ens}
    \stackrel{\mathcal{F}_{\Gamma, G}}{\longrightarrow} \mbox{Ens}
    \stackrel{\imath}{\longrightarrow}
    \widetilde{\mbox{Ens}}
\]
where $\times$ is the Cartesian product functor, $\sqcup$ is the (disjoint) union functor,
and $\imath$ is the natural inclusion functor.  That is,
$\mathcal{F}_{\Gamma, G} \times \mathcal{F}_{\Gamma, G} (\Omega_1, \Omega_2)$
is the finite set of pairs $(P_1, P_2)$ where
$P_1 \in \mathcal{F}_{\Gamma, G} (\Omega_1)$
is a $\Gamma$-$G$-principal bundle over $\Omega_1$ with total space $\Omega_1 \times G$,
$P_2 \in \mathcal{F}_{\Gamma, G} (\Omega_2)$
is a $\Gamma$-$G$-principal bundle over $\Omega_2$ with total space $\Omega_2 \times G$, and
$(\mathcal{F}_{\Gamma, G} \times \sqcup) (\Omega_1, \Omega_2)$
is the finite set of $\Gamma$-$G$-principal bundles in
$\mathcal{F}_{\Gamma, G}(\Omega_1 \sqcup \Omega_2)$.

Define a natural transformation
$\eta\co\mathcal{F}_{\Gamma, G}\times\mathcal{F}_{\Gamma, G}
\to \mathcal{F}_{\Gamma, G}\times\sqcup$
as follows.  To each pair $(\Omega_1, \Omega_2)$ of finite sets, we assign the morphism
\[
    (\mathcal{F}_{\Gamma, G} \times \mathcal{F}_{\Gamma, G}) (\Omega_1, \Omega_2)
    \stackrel{\eta_{(\Omega_1, \Omega_2)}}{\longrightarrow}
    (\mathcal{F}_{\Gamma, G} \times \sqcup) (\Omega_1, \Omega_2)
\]
such that $\eta_{(\Omega_1, \Omega_2)}(P_1, P_2) = P_1 \sqcup P_2$
as a $\Gamma$-$G$-principal bundle over $\Omega_1 \sqcup \Omega_2$ with total space $(\Omega_1 \sqcup \Omega_2) \times G$.
It is straightforward to show that $\eta$ is a weak decomposition of the functor
$\mathcal{F}_{\Gamma, G}$ as defined in \cite[page 192]{dressmuller}.
As well, we define $\mathcal{F}_{\Gamma, G}^\eta(\Omega)$ to be the collection
of $\Gamma$-$G$-principal bundles in $\mathcal{F}_{\Gamma, G}(\Omega)$ that are not in the image
of $\eta$ for some partition $\Omega = \Omega_1 \sqcup \Omega_2$ with
$\Omega_1, \Omega_2 \neq \emptyset$, i.e.
\[
    \mathcal{F}_{\Gamma, G}^\eta(\Omega)
    =
    \begin{cases}
        \mathcal{F}_{\Gamma, G}(\Omega) \smallsetminus
        \bigcup\limits_{\substack{\Omega=\Omega_1 \sqcup \Omega_2, \\ \Omega_1, \Omega_2 \neq \emptyset}}
        \eta(\mathcal{F}_{\Gamma, G}(\Omega_1) \times \mathcal{F}_{\Gamma, G}(\Omega_2)),
                    &       \Omega \neq \emptyset               \\
        \emptyset,  &       \Omega = \emptyset.
            \end{cases}
\]
Then $\mathcal{F}_{\Gamma, G}^\eta(\Omega)$ is the collection of \emph{irreducible}
$\Gamma$-$G$-bundles in $\mathcal{F}_{\Gamma, G}(\Omega)$.

Recall that $\mathbf{n}$ denotes the set $\{ 1, 2, \ldots, n \}$.
Given a pair $(P_1, P_2)$ of $\Gamma$-$G$-principal bundles over
$\mathbf{r}$ and $\mathbf{s}$, respectively, with $n = r + s$ and identifying
$\mathbf{n}$ with $\mathbf{r} \sqcup \mathbf{s}$, choose homomorphisms
$\theta_1 \in HOM(\Gamma, G(\mathcal{S}_r))$ and
$\theta_2 \in HOM(\Gamma, G(\mathcal{S}_s))$ such that $(\theta_i)$ is the conjugacy
class associated to the isomorphism class $[P_i]$ of $P_i$ by Theorem \ref{th-CrossedPdctBundleChr}({\it i}).
Then the $\Gamma$-$G$-principal bundle
$\eta_{(\Omega_1, \Omega_2)}(P_1, P_2)$
corresponds to the homomorphism
$\theta\co \Gamma \to G(\mathcal{S}_r) \times G(\mathcal{S}_s) \leq G(\mathcal{S}_{r+s})$
given by $\theta(\gamma) = \theta_1(\gamma)\theta_2(\gamma)$, where
$\theta_1(\gamma)$ and $\theta_2(\gamma)$ are considered elements of the first
and second factors, respectively, of
$G(\mathcal{S}_r) \times G(\mathcal{S}_s) \leq G(\mathcal{S}_{r+s})$.
Note that the fixed point set of $\theta$ in $M^{r+s}$ is
\begin{equation}
\label{eq-DMFixedPoint}
    (M^{r+s})^{\langle\theta\rangle} =
    (M^r)^{\langle\theta_1\rangle} \times (M^s)^{\langle\theta_2\rangle}.
\end{equation}
The sets in Equation \eqref{eq-DMFixedPoint} of course depend on the choice of $\theta_1$ and $\theta_2$,
but their diffeomorphism-type depends only on the corresponding conjugacy classes
and hence on the isomorphism classes of the bundles $P_1$ and $P_2$.

For a global quotient orbifold $M \rtimes G$ where $M$ is a closed manifold,
define  the formal power series
\begin{align*}
    \Phi(q)
    &=      \sum\limits_{n \geq 1}\;
            \sum\limits_{P \in \mathcal{F}_{\Gamma, G}(\mathbf{n})}  \;
                \frac{q^n}{n! |G|^n}
            \chi \left((M^n)^{\langle \theta_P \rangle}\right)
            \quad\mbox{and}
    \\
    \Psi(q)
    &=      1 + \sum\limits_{n \geq 1} \;
            \sum\limits_{P \in \mathcal{F}_{\Gamma, G}^\eta(\mathbf{n})} \;
                \frac{q^n}{n! |G|^n}
            \chi \left((M^n)^{\langle \theta_P \rangle}\right).
\end{align*}
Note that $\chi \left((M^n)^{\langle \theta_P \rangle}\right)$ depends only on the conjugacy class of $\theta_P$
so that both series are well defined.

Now, to link $\Phi$ and $ \Psi$  with the invariants
introduced in Definition \ref{def-multInvAssocInvariants}, define the functions
\begin{align*}
    \phi_{\Gamma, M\rtimes G}^\eta    &\co \Z_{\geq 0}  \longrightarrow   \Q \quad\mbox{and}        \\
    \psi_{\Gamma, M\rtimes G}         &\co \Z_{\geq 0}  \longrightarrow   \Q
\end{align*}
by setting
\begin{align*}
    \phi_{\Gamma, M\rtimes G}^\eta(n)
        &= \sum\limits_{P \in \mathcal{F}_{\Gamma, G}^\eta(\mathbf{n})}
            \frac{1}{|G|^n}\chi \left((M^n)^{\langle \theta_P \rangle}\right)
        \quad\quad\mbox{and}
        \\
    \psi_{\Gamma, M\rtimes G}(n)
        &= \sum\limits_{P \in \mathcal{F}_{\Gamma, G}(\mathbf{n})}
            \frac{1}{|G|^n}\chi \left((M^n)^{\langle \theta_P \rangle}\right).
\end{align*}
As a convention, we set $\phi_{\Gamma, M\rtimes G}^\eta(0) = 0$ and $\psi_{\Gamma, M\rtimes G}(0) = 1$.
Then
\[
    \Phi(q) =   \sum\limits_{n\geq 0} \phi_{\Gamma, M\rtimes G}^\eta(n) \frac{q^n}{n!}
    \;\;\;\mbox{and}\;\;\;
    \Psi(q) =   \sum\limits_{n\geq 0} \psi_{\Gamma, M\rtimes G}(n) \frac{q^n}{n!}.
\]
The relationships between the series $\Phi$ and $\Psi$ and extensions of the Euler--Satake characteristic
are indicated by the simple computations below.

The $G(\mathcal{S}_n)$-conjugacy class $(\theta)$ of a $\theta\co\Gamma\to G(\mathcal{S}_n)$
contains
$\frac{|G(\mathcal{S}_n)|}{|C_{G(\mathcal{S}_n)}(\theta)|} = \frac{|G|^n n!}{|C_{G(\mathcal{S}_n)}(\theta)|}$
elements.  Hence,
\begin{align}
    \nonumber
    \Psi(q)
    &=
        \sum\limits_{n\geq 0} \frac{q^n}{n!}
        \sum\limits_{P \in \mathcal{F}_{\Gamma, G}(\mathbf{n})}
            \frac{1}{|G|^n}\chi \left((M^n)^{\langle \theta_P \rangle}\right)
        \\ \nonumber
    &=
        \sum\limits_{n\geq 0} \frac{q^n}{n!}
        \sum\limits_{(\theta) \in HOM(\Gamma, G(\mathcal{S}_n))/G(\mathcal{S}_n)}\;\;\;
        \sum\limits_{\tau \in (\theta)}
            \frac{1}{|G|^n}\chi \left((M^n)^{\langle \tau \rangle}\right)
        \\ \label{eq-DMPsi}
    &=
        \sum\limits_{n\geq 0} \frac{q^n}{n!}
        \sum\limits_{(\theta) \in HOM(\Gamma, G(\mathcal{S}_n))/G(\mathcal{S}_n)}
            \left(\frac{|G|^n n!}{|C_{G(\mathcal{S}_n)}(\theta)|}\right)
            \frac{1}{|G|^n}\chi \left((M^n)^{\langle \theta \rangle}\right)
        \\ \nonumber
    &=
        \sum\limits_{n\geq 0} q^n
        \sum\limits_{(\theta) \in HOM(\Gamma, G(\mathcal{S}_n))/G(\mathcal{S}_n)}
            \frac{\chi \left((M^n)^{\langle \theta \rangle}\right)}{|C_{G(\mathcal{S}_n)}(\theta)|}
        \\ \nonumber
    &=
        \sum\limits_{n\geq 0}q^n
        \sum\limits_{(\theta) \in HOM(\Gamma, G(\mathcal{S}_n))/G(\mathcal{S}_n)}
            \chi_{ES}\left((M^n)^{\langle \theta \rangle} \rtimes C_{G(\mathcal{S}_n)}(\theta)\right)
        \\ \nonumber
    &=
        \sum\limits_{n\geq 0}q^n\chi_\Gamma^{ES}(M^n \rtimes (G(\mathcal{S}_n))).
\end{align}

The same computation shows that
\[
    \Phi(q)
    =
        \sum\limits_{n\geq 0}q^n
        \sum\limits_{(\theta) \in HOM(\Gamma, G(\mathcal{S}_n))^t/G(\mathcal{S}_n)}
            \chi_{ES}\left((M^n)^{\langle \theta \rangle} \rtimes C_{G(\mathcal{S}_n)}(\theta)\right),
\]
where $HOM(\Gamma, G(\mathcal{S}_n))^t/G(\mathcal{S}_n)$ denotes the conjugacy classes
$(\theta)$ of homomorphisms such that the image of $\pi(\theta)$ acts transitively on $\mathbf{n}$,
or equivalently such that the associated isomorphism class of $\Gamma$-$G$-principal bundles is irreducible.
As isomorphism classes of finite, transitive $\Gamma$-sets correspond to conjugacy classes of
subgroups $H \leq \Gamma$ of finite index, this becomes
\[
        \sum\limits_{n\geq 0}q^n
        \sum\limits_{(H_n)}\;\;
        \sum\limits_{(\theta) \in \pi^{-1}((\Gamma/H))}
            \chi_{ES}\left((M^n)^{\langle \theta \rangle} \rtimes C_{G(\mathcal{S}_n)}(\theta)\right),
\]
where the second sum ranges over all $\Gamma$-conjugacy classes $(H_n)$ of subgroups $H_n \leq \Gamma$ of index
$n$ and $\pi \co HOM(\Gamma, G({\mathcal S}_n))/ G({\mathcal S}_n ) \to HOM(\Gamma, {\mathcal S}_n)/ {\mathcal S}_n$
is as in Proposition \ref{prop-MultInvarGamSetX}; by the same Proposition, this is equal to
\begin{equation}
\label{eq-DMPhi}
    \Phi(q)
    =
    \sum\limits_{n\geq 1}q^n
    \sum\limits_{(H): |\Gamma/H|=n}\;\;
    \chi_{(\Gamma/H)}^{ES} (M \rtimes G).
\end{equation}

We are now ready to prove the following.

\begin{theorem}
\label{th-DressMuellerGen}
Let $\Gamma$ be a finitely generated discrete group, let $M$ be a compact manifold,
and let $M\rtimes G$ be a presentation of the global quotient orbifold $Q$ so that $G$ is finite.
Then with the definitions given above, we have
\[
    \Psi(q) =   \exp(\Phi(q)),
\]
i.e., applying Equations \eqref{eq-DMPsi} and \eqref{eq-DMPhi},
\begin{equation}
\label{eq-DMFormula}
        \sum\limits_{n\geq 0}q^n\chi_\Gamma^{ES}( M^n \rtimes G(\mathcal{S}_n)) =
        \exp\left( \sum\limits_{n\geq 1}q^n
            \sum\limits_{(H): |\Gamma/H|=n}\;\; \chi_{(\Gamma/H)}^{ES} ( M \rtimes G)\right).
\end{equation}
\end{theorem}

\begin{remark}
Note that if we apply \cite[Theorem 1]{dressmullerII} choosing the weight $\omega$
to be the Euler--Satake characteristic, we have
\[
    \sum_{[X]}\, q^{|X|}\, \chi_{[X]}^{ES} (M\rtimes G) =
    \exp\left( \sum\limits_{n\geq 1}q^n
    \sum\limits_{(H): [\Gamma:H]=n}\;\; \chi_{(\Gamma/H)}^{ES} ( M \rtimes G)\right).
\]
Along with Equation \eqref{eq-GenGammaSettoGammaExt}, this as well yields
Equation \eqref{eq-DMFormula}.  Recall that $\chi_{(\Gamma/H)}^{ES}$ corresponds to the application of
$\chi^{ES}$ to the groupoid $\mathcal{A}$ defined in Section \ref{subsec-wrlimynewIIItheHinertia}.
\end{remark}

\begin{proof}
This result is an analog of \cite[Equation (1.6)]{dressmuller}, which is proven for
an arbitrary decomposable functor $\mathcal{F}$, but defining the functions
$\phi_{\Gamma, M\rtimes G}$ and $\psi_{\Gamma, M\rtimes G}$ to be the
counting functions for finite sets.  Here, we illustrate that Dress and M\"uller's arguments
apply to extensions of the Euler--Satake characteristic as defined above.
Their proof of this result is separated into
parts (i), (iii), (iv), (v), (vii), (viii), and (xi); only (xi) refers to the counting
functions (note that (ii), (vi), and (ix) are used to prove a separate result).
As our functors $\mathcal{F}_{\Gamma, G}$ and $\mathcal{F}_{\Gamma, G}^\eta$
are special cases of theirs, their results apply and we need only verify (xi).

By \cite[(vii) and (viii)]{dressmuller}, for any finite set $\Omega$ and any fixed element
$\omega \in \Omega$, we have
\[
    \mathcal{F}_{\Gamma, G}(\Omega)
    =
    \bigcup\limits_{\Omega_1\co\omega \in \Omega_1 \subseteq \Omega}
    \eta\left(\mathcal{F}_{\Gamma, G}^\eta(\Omega_1)
    \times \mathcal{F}_{\Gamma, G}(\Omega \smallsetminus \Omega_1)\right),
\]
and the right-side of this equation is a disjoint union. As
each $\eta_{(\Omega_1, \Omega\smallsetminus\Omega_1)}$ is
injective, we can use this decomposition, Equation
\eqref{eq-DMFixedPoint}, and the multiplicativity of $\chi$ to
rewrite $\psi_{\Gamma, M\rtimes G}(n)$ as
\begin{align*}
        &       \sum\limits_{P \in \mathcal{F}_{\Gamma, G}(\mathbf{n})}
                \frac{1}{|G|^n}\chi \left((M^n)^{\langle \theta_P \rangle}\right)             \\
        &\quad=      \sum\limits_{\Omega_1 : 1 \in \Omega_1 \subseteq \mathbf{n}} \;\;\;
                \sum\limits_{P \in \eta(\mathcal{F}_{\Gamma, G}^\eta(\Omega_1)
                \times \mathcal{F}_{\Gamma, G}(\mathbf{n} \smallsetminus \Omega_1))}
                \frac{1}{|G|^n}\chi \left((M^n)^{\langle \theta_P \rangle}\right)             \\
        &\quad=      \sum\limits_{\Omega_1 : 1 \in \Omega_1 \subseteq \mathbf{n}} \;\;\;
                \sum\limits_{P_1 \in \mathcal{F}_{\Gamma, G}^\eta(\Omega_1)}\;\;\;
                \sum\limits_{P_2 \in \mathcal{F}_{\Gamma, G} (\mathbf{n}\smallsetminus \Omega_1)}
                \frac{1}{|G|^n}\chi \left((M^n)^{\langle \theta_{P_1}\theta_{P_2} \rangle}\right)   \\
        &\quad=      \sum\limits_{\Omega_1 : 1 \in \Omega_1 \subseteq \mathbf{n}} \;\;\;
                \sum\limits_{P_1 \in \mathcal{F}_{\Gamma, G}^\eta(\Omega_1)}
                \frac{1}{|G|^\mu}\chi \left((M^\mu)^{\langle \theta_{P_1} \rangle}\right)       \\
        &\quad\quad\quad
                \sum\limits_{P_2 \in \mathcal{F}_{\Gamma, G} (\mathbf{n}\smallsetminus \Omega_1)}
                \frac{1}{|G|^{n-\mu}}\chi \left((M^{n-\mu})^{\langle \theta_{P_2} \rangle}\right)   \\
        &\quad       \mbox{(where $\mu$ is the cardinality of $\Omega_1$)}                                   \\
        &\quad=      \sum\limits_{\Omega_1 : 1 \in \Omega_1 \subseteq \mathbf{n}} \;
                \phi^\eta(\mu) \psi(n-\mu).
\end{align*}
In particular, the expression
$\phi_{\Gamma, M\rtimes G}^\eta(\mu) \psi_{\Gamma, M\rtimes G}(n - \mu)$ depends only
on the cardinality of $\Omega_1$.

For each $\mu$ with $1 \leq \mu \leq n$, the $\binom{n-1}{\mu-1}$ subsets
of $\mathbf{n}$ of cardinality $\mu$ containing the element $1$ contribute
$\binom{n-1}{\mu-1}\phi_{\Gamma, M\rtimes G}^\eta(\mu)\psi_{\Gamma, M\rtimes G}(n - \mu)$,
and
\[
    \psi_{\Gamma, M\rtimes G}(n)
    =
    \sum\limits_{\mu=1}^n \binom{n-1}{\mu-1}
    \phi_{\Gamma, M\rtimes G}^\eta(\mu)\psi_{\Gamma, M\rtimes G}(n - \mu)
\]
for $n \geq 1$.  Multiplying both sides by $q^{n-1}/(n-1)!$ and summing over $n \geq 1$ yields
\begin{align*}
    \sum\limits_{n\geq 1} \frac{q^{n-1}}{(n-1)!} \psi_{\Gamma, M\rtimes G}(n)
    &=
        \sum\limits_{n\geq 1} \frac{q^{n-1}}{(n-1)!}
        \sum\limits_{\mu=1}^n \binom{n-1}{\mu-1}
        \phi_{\Gamma, M\rtimes G}^\eta(\mu)\psi_{\Gamma, M\rtimes G}(n - \mu)
            \\
    &=
    \sum\limits_{n\geq 1} q^{n-1}
    \sum\limits_{\mu=1}^n \frac{1}{(\mu - 1)!(n - \mu)!}
    \phi_{\Gamma, M\rtimes G}^\eta(\mu)\psi_{\Gamma, M\rtimes G}(n - \mu),
\end{align*}
and hence
\[
    \Psi^\prime(q) = \Phi^\prime(q)\Psi(q).
\]
By $\Psi^\prime(q)$ and $\Phi^\prime(q)$, we mean the formal derivatives of the corresponding
power series.  Recalling that
$\psi_{\Gamma, M\rtimes G}(0) = 1$ and $\phi_{\Gamma, M\rtimes G}^\eta(0) = 0$,
the claim follows.
\end{proof}

As an application, we demonstrate how Theorem \ref{th-GenerFunctGammaExtensEulerSatakeCharacteristic}
follows from Theorem \ref{th-DressMuellerGen} in the case that $G$ is finite.

\begin{proof}[Alternate proof of Theorem \ref{th-GenerFunctGammaExtensEulerSatakeCharacteristic} for $G$ finite]
Combining Theorem \ref{th-DressMuellerGen} and Definition \ref{def-multInvAssocInvariants}(ii), we have
\begin{align}
    \nonumber
    &\sum\limits_{n\geq 0}q^n\chi_\Gamma^{ES}( M^n \rtimes G(\mathcal{S}_n))
        \\ \label{eq-AltProof1}
        &\quad\quad=
        \exp\left( \sum\limits_{n\geq 1}q^n
        \sum\limits_{(H): |\Gamma/H|=n}\;\; \chi_{(\Gamma/H)}^{ES} ( M \rtimes G)\right)
        \\ \nonumber
        &\quad\quad=
        \exp\left( \sum\limits_{n\geq 1} q^n
            \sum\limits_{(H): |\Gamma/H|=n} \;\;  \sum\limits_{[\tau]}\;\;
        \left(  M^{\langle\tau\rangle}\rtimes Aut_{\Gamma,\,G}^{P_\tau}  \right)\right)
\end{align}
where the sum over $[\tau]$ ranges over $HOM(H, G)/(N_{\Gamma}(H) \times G)$.
Recall from Equation \eqref{eq-CG=Aut} that $Aut_{\Gamma,\,G}^{P_\tau} \cong  C_{G({\mathcal S}_n)}(\tau)$.
As in the original proof of Theorem
\ref{th-GenerFunctGammaExtensEulerSatakeCharacteristic}, we
have by \cite[Theorem 4-2]{tamanoi2} that
the index $[Aut_{\Gamma,\,G}^{P_\tau} :C_G(\tau)]$ is
given by $|N_\Gamma^\tau(H)/H|$, which is finite.  It follows that
\[
    M^{\langle \tau\rangle}\rtimes C_G(\tau)
    \longrightarrow
    M^{\langle \tau\rangle}\rtimes Aut_{\Gamma,\,G}^{P_\tau}
\]
is an orbifold cover of $|N_\Gamma^\tau(H)/H|$ sheets, so we may rewrite Equation
\eqref{eq-AltProof1} as
\[
    \exp\left( \sum\limits_{n\geq 1}q^n
        \sum\limits_{(H): |\Gamma/H| = n}\;\;  \sum\limits_{[\tau]}\;\; \frac{1}{|N_\Gamma^\tau(H)/H|}
        \left(  M^{\langle\tau\rangle}\rtimes C_{G}(\tau) \right)\right).
\]
The $N_\Gamma(H)$-orbit of each $\tau \in HOM(H, G)$ has $|N_\Gamma(H)/N_\Gamma^\tau(H)|$ elements,
so that summing over $G$-conjugacy classes $(\tau) \in HOM(H, G)/G$ rather than
$N_\Gamma(H) \times G$-conjugacy classes $[\tau]$, we have
\[
        \exp\left( \sum\limits_{n\geq 1}q^n\
            \sum\limits_{(H): |\Gamma/H|=n}\; \frac{1}{|N_\Gamma(H)/H|} \;  \sum\limits_{(\tau)}\;\;
            \left(  M^{\langle\tau\rangle}\rtimes C_{G}(\tau) \right)\right).
\]
As each conjugacy class $(H)$ contains $|\Gamma/N_\Gamma(H)|$ elements, we rewrite this as
\begin{align*}
    &\exp\left( \sum\limits_{n\geq 1}q^n\
        \sum\limits_{H: |\Gamma/H|=n} \frac{1}{|{\Gamma}/H|}  \sum\limits_{(\tau)}\;\ \chi^{ES}
        \left(  M^{\langle\tau\rangle}\rtimes C_{G}(\tau)  \right)\right)
    \\&\quad\quad=
    \exp\left( \sum\limits_{n\geq 1}\frac{q^n}{n}\
            \sum\limits_{H: |\Gamma/H|=n}  \ \chi^{ES}_H
\left(  M^{\langle\tau\rangle}\rtimes C_{G}(\tau)  \right)\right),
\end{align*}
completing the proof.
\end{proof}

\begin{remark}
It is likely that the above approach can be used to derive formulas for numerical invariants of
wreath products of finite group quotients of quasi-projective algebraic varieties, c.f. \cite[Theorem 4]{glmold}
and \cite[Theorem 1]{glm}.
\end{remark}

\begin{remark}
The main challenge in extending this approach to the case where $G$ is an infinite compact Lie group
is that the set of $\Gamma$-$G$-bundles over a finite set $\Omega$ is no longer finite.  Hence
the functor $\mathcal{F}_{\Gamma, G}: Ens \to Ens$ would need to assign to $\Omega$
the finite set of isomorphisms classes of $\Gamma$-$G$-bundles over $\Omega$, requiring isomorphisms
that restrict to the identity on $\Omega$.  But then the natural transformation $\eta$
is no longer a weak decomposition.
\end{remark}


\section*{Acknowledgements}

The first author would like to thank the MSRI and the Center for Interdisciplinary Mathematical Sciences
Banaras Hindu University, Varanasi, India for its hospitality during the preparation of this manuscript.
The second author was supported by a Rhodes College Faculty Development Endowment Grant.


\bibliographystyle{amsplain}

\begin{thebibliography}{10}

\bibitem{ademleidaruan}
A. Adem, J. Leida, and Y. Ruan:
Orbifolds and Stringy Topology,
Cambridge Tracts in Mathematics \textbf{171},
Cambridge University Press, Cambridge, 2007.

\bibitem{planes}
M. Abreu, E.  Dryden, P.  Freitas, and L. Godinho:
\emph{Hearing the weights of weighted projective planes},
Ann. Global Anal. Geom. \textbf{33},  (2008), 373--395.

\bibitem{bfnr}
A. Bahri, M. Franz, D. Notbohm, N. Ray:
\emph{The classification of weighted projective spaces},
Fund. Math. \textbf{220}, (2013), 217--226.

\bibitem{boli}
L. A. Borisov and A. Libgober:
\emph{Elliptic genera of toric varieties and applications to mirror symmetry},
Invent. Math. \textbf{140}, (2000), 453--485.


\bibitem{bryanfulman}
J. Bryan and J. Fulman:
\emph{Orbifold {E}uler characteristics and the number of commuting $m$-tuples in the symmetric groups},
Ann. Comb. \textbf{2}, (1998), 1--6.

\bibitem{dressmullerII}
P. J.  Cameron, C. Krattenthaler and   T. M\"uller:
\emph{Decomposable Functors and the Exponential Principle II},
S\'em. Lothar. Combin. \textbf{61A} (2009/10), Art. B61Am, 38 pp.

\bibitem{dmvv}
R. Dijkgraaf, G. Moore, E. Verlinde and H.  Verlinde:
\emph{Elliptic genera of symmetric products and second quantized strings},
Comm. Math. Phys.  \textbf{185}  (1997), 197--209.

\bibitem{dixon}
L. Dixon, J. Harvey, Vafa, and E. Witten:
\emph{Strings on orbifolds},
Nucl. Phys. B \textbf{261}, (1985), 678--686.

\bibitem{dressmuller}
A. Dress and T. M\"uller:
\emph{Decomposable Functors and the Exponential Principle},
Adv. Math.  \textbf{129} (1997), 188--221.

\bibitem{farsiseaton1}
C. Farsi and C. Seaton:
\emph{Nonvanishing vector fields on orbifolds},
Trans. Amer. Math. Soc. \textbf{362} (2010), 509--535.

\bibitem{farsiseaton2}
C. Farsi and C. Seaton:
\emph{Generalized twisted sectors of orbifolds},
Pacific J. Math. \textbf{246} (2010), 49--74.

\bibitem{farsiseaton3}
C. Farsi and C. Seaton:
\emph{Generalized orbifold Euler characteristic of general orbifolds and wreath products},
Algebr. Geom. Topol. \textbf{11} (2011), 523--551.

\bibitem{dlm}
C. Dong, K. Liu,  X. Ma:
On orbifold elliptic genus, Orbifolds in mathematics and physics (Madison, WI, 2001), 87--105,
Contemp. Math., \textbf{310}, Amer. Math. Soc., Providence, RI, 2002.

\bibitem{hh}
F. Hirzebruch and T.  H\"ofer:
\emph{On the Euler number of an orbifold},
Math. Ann. \textbf{286} (1990), 255--260.

\bibitem{glmold}
S.M. Gusein-Zade, I. Luengo and A. Melle--Hern\'andez:
\emph{On the power structure over the Grothendieck ring of varieties and its applications},
Tr. Mat. Inst. Steklova \textbf{258} (2007), Anal. i Osob. Ch. 1, 58--69; translation in
Proc. Steklov Inst. Math. \textbf{258} (2007), 53--64

\bibitem{glm}
S.M. Gusein-Zade, I. Luengo and A. Melle--Hern\'andez:
\emph{Higher order generalized Euler characteristics
and generating series}, arxiv:1303.5574.

\bibitem{hambletonhausmann}
I. Hambleton and  J. C. Hausmann:
\emph{Equivariant bundles and isotropy representations},
Groups Geom. Dyn. \textbf{4} (2010), 127--162.

\bibitem{hao}
A. Hattori:
\emph{Orbifold elliptic genera and rigidity}.
J. Math. Soc. Japan \textbf{58} (2006), 419--452.

\bibitem{hama}
A. Hattori and M. Masuda:
\emph{Elliptic genera, torus orbifolds and multi-fans}, Intern. J. Math., \textbf{16} (2005), 957--998.

\bibitem{holm}
T. Holm:
\emph{Orbifold cohomology of abelian symplectic reductions and the case of weighted projective spaces},
Poisson geometry in mathematics and physics, 127--146,
Contemp. Math., \textbf{450}, Amer. Math. Soc., Providence, RI, 2008.

\bibitem{husemoller}
D. Husemoller:
\emph{ Fibre Bundles}, Third edition,
Graduate Texts in Mathematics, \textbf{20}, Springer--Verlag, New York, 1994.

\bibitem{jiang}
Y. Jiang:
\emph{The Chen-Ruan cohomology of weighted projective spaces,}
Canad. J. Math. \textbf{59} (2007), 981--1007.

\bibitem{kawasaki}
T. Kawasaki:
\emph{Cohomology of twisted projective spaces and lens complexes},
Math. Ann. \textbf{206} (1973), 243--248.

\bibitem{lashof1}
R. Lashof:
\emph{Equivariant bundles},
Illinois J. Math. \textbf{26} (1982), 257--271.

\bibitem{moerdijkmrcun}
I. Moerdijk and J. Mr\v{c}un:
Introduction to Foliations and {L}ie groupoids,
Cambridge Studies in Advanced Mathematics \textbf{91},
Cambridge University Press, Cambridge, NY, 2003.

\bibitem{oh}
T. Ohmoto:
\emph{Generating functions of orbifold Chern classes. I. Symmetric products},
Math. Proc. Cambridge Philos. Soc.  \textbf{144}  (2008), 423--438.

\bibitem{tamanoi1}
H. Tamanoi:
\emph{Generalized orbifold {E}uler characteristic of
symmetric products and equivariant {M}orava $K$-theory},
Algebr. Geom. Topol. \textbf{1} (2001), 115--141.

\bibitem{tamanoi2}
H. Tamanoi:
\emph{Generalized orbifold {E}uler characteristic of symmetric orbifolds and covering spaces},
Algebr. Geom. Topol. \textbf{3} (2003), 791--856.

\bibitem{tamanoi3}
H. Tamanoi:
\emph{Infinite product decomposition of orbifold mapping spaces},
Algebr. Geom. Topol. \textbf{9} (2009), 569--592.

\bibitem{thurston}
W. Thurston: The Geometry and Topology of $3$-Manifolds,
Lecture Notes, Princeton University Math Dept., Princeton, New Jersey (1978).

\bibitem{wang}
W. Wang:
\emph{Equivariant $K$-theory, wreath products, and Heisenberg algebra},
Duke Math. J. \textbf{103} (2000), 1--23.

\bibitem{wangZhou}
W. Wang and J. Zhou:
\emph{Orbifold Hodge numbers of the wreath product orbifolds},
J. Geom. Phys. \textbf{38} (2001), 152-–169.

\end{thebibliography}

\end{document}